\documentclass[12pt,english]{article}
\usepackage{amsmath,amssymb,amsthm}
\usepackage{color,xcolor,times,fullpage}
\usepackage{graphicx,eulervm,makecell,multirow}
\graphicspath{{./BC_graphs/} }
\usepackage[colorlinks,linkcolor=blue,citecolor=cyan]{hyperref}
\usepackage[skip=2pt]{caption}
\usepackage{mathtools,multirow}
\usepackage[format=plain,labelfont={bf,it},textfont=it]{caption}
\numberwithin{equation}{section}

\title{\textbf{Periodic minimum in the count of binomial coefficients 
not divisible by a prime}}

\author{Hsien-Kuei Hwang \\
    Institute of Statistical Science \\
    Academia Sinica\\
    Taipei 115\\
    Taiwan
\and Svante Janson\\ 
    Department of Mathematics\\ 
    Uppsala University\\ 
    Uppsala\\ 
    Sweden
\and Tsung-Hsi Tsai \\
    Institute of Statistical Science \\
    Academia Sinica\\
    Taipei 115\\
    Taiwan}
\date{\today} 

\theoremstyle{plain}
\newtheorem{theorem}{Theorem}[section]
\newtheorem{lemma}[theorem]{Lemma}

\newtheorem{corollary}[theorem]{Corollary}
\newtheorem{problem}[theorem]{Problem}
\theoremstyle{definition}

\newcommand\xqed[1]{%
    \leavevmode\unskip\penalty9999 \hbox{}\nobreak\hfill
    \quad\hbox{#1}}

\newtheorem{exampleqqq}[theorem]{Example}

\newtheorem{remarkqqq}[theorem]{Remark}
\newenvironment{remark}{\begin{remarkqqq}}
  {\xqed{$\triangle$}\end{remarkqqq}}

\newcommand{\refT}[1]{Theorem~\ref{#1}}

\newcommand{\refL}[1]{Lemma~\ref{#1}}

\newcommand{\refR}[1]{Remark~\ref{#1}}

\newcommand{\refS}[1]{Section~\ref{#1}}

\newcommand\set[1]{\ensuremath{\{#1\}}}

\newcommand\bigpar[1]{\bigl(#1\bigr)}
\newcommand\Bigpar[1]{\Bigl(#1\Bigr)}

\newcommand\lrpar[1]{\left(#1\right)}

\def\rompar(#1){\textup(#1\textup)}    

\def\xexp(#1){e^{#1}}

\newcommand\floor[1]{\lfloor#1\rfloor}

\newcommand\frax[1]{\{#1\}}
\newcommand\ntoo{\ensuremath{{n\to\infty}}}

\newcommand\bbR{\mathbb R}

\newcounter{CC}
\newcounter{cc}

\newcommand\gb{\beta}
\newcommand\gd{\delta}
\newcommand\gD{\Delta}
\newcommand\gf{\varphi}

\renewcommand\phi{\xxx}  

\newcommand\cP{\mathcal P}
\newcommand\tb{\tilde b}

\newcommand\qw{^{-1}}

\newcommand\oi{\ensuremath{[0,1]}}

\newcommand\hb{{\hat{b}}}
\newcommand\hs{{\hat s}}
\newcommand\rhop{{\varrho_p}}

\newcommand\eqtext[1]{\quad\text{#1}\quad}

\begin{document}

\maketitle
\begin{abstract}

The summatory function of the number of binomial coefficients not
divisible by a prime is known to exhibit regular periodic
oscillations, yet identifying the less regularly behaved minimum of
the underlying periodic functions has been open for almost all cases.
We propose an approach to identify such minimum in some generality,
solving particularly a previous conjecture of B.~Wilson [Asymptotic
behavior of Pascal’s triangle modulo a prime, \emph{Acta Arith.}
\textbf{83} (1998), pp.~105--116].

\end{abstract}

\maketitle

\section{Introduction }

Let $F_p(n)$ denote the number of binomial coefficients
$\binom{m}{k}$, $0\leq k\leq m<n$, that are not divisible by a given
prime $p$. In particular, for $p=2$, $F_2(n)$ is the number of odd
numbers in the first $n$ rows of Pascal's triangle. The study of the
quantity $F_p(n)$ has a long history; see, for example, the
historical account in Stolarsky's paper \cite{sto}. Some sequences of
$F_p(n)$ appear in the On-Line Encyclopedia of Integer Sequences
(OEIS) \cite{oe}: A006046 ($p=2$), A006048 ($p=3$), and A194458
($p=5$).

Fine \cite{fi} proved that ``almost all'' binomial coefficients are
divisible by a prime $p$, more precisely that
\begin{equation}
	\lim_{n\rightarrow\infty}
	\frac{F_p(n)}{\binom{n+1}{2}}=0.
	\label{fine1}
\end{equation}
He \cite{fi} also gave the expression (with $n=\sum_{0\le i\le s} b_i2^i$, $b_i\in\{0,1\}$)
\begin{equation}
	F_p(n+1)-F_p(n)=\prod_{0\leq i\leq s}(b_i+1)
	\label{fine2}
\end{equation}
for the number of binomial coefficients $\binom nk$, $0\le k\le n$, 
not divisible by $p$.

Later Stein \cite{st} observed that, for any prime $p$,
\begin{align}\label{stein}
  F_p(pn)=\binom{p+1}2F_p(n),
\qquad n\ge1,
\end{align}
Thus the sequence $\psi_p(n):=F_p(n)/n^\rhop$, where
$\varrho_p=\log_p\binom{p+1}{2}$, satisfies $\psi_p(pn)=\psi_p(n)$,
so that $\psi_p$ can be extended by this property to all positive
$p$-adic rational numbers. Stein \cite{st} also showed that $\psi_p$
can be further extended to a continuous function on $(0,\infty)$; in 
other words, there exists a continuous 1-periodic function $\cP_p(t)$ 
on $\bbR$ such that
\begin{align}\label{Pp}
	F_p(n)=n^\rhop \cP_p(\log_p n),
	\qquad n\ge1. 
\end{align}
It follows immediately from this that
\begin{align}\label{gap}
    \alpha_p
	&:=\limsup_\ntoo\frac{F_p(n)}{n^\rhop}
	=\sup_{n\ge1}\frac{F_p(n)}{n^\rhop}
	=\max_{t\in\oi}\cP_p(t)
	\in[0,\infty),
	\\ \label{gbp}
    \gb_p&:=\liminf_\ntoo\frac{F_p(n)}{n^\rhop}
	=\inf_{n\ge1}\frac{F_p(n)}{n^\rhop}
	=\min_{t\in\oi}\cP_p(t)
	\in[0,\infty),
\end{align}
furthermore, $\alpha_p=1$ for every $p$, and that $\binom{p+1}2\qw \le
\gb_p<1$; see \cite{st}. The extremal properties of $\cP_2$ had
earlier been treated by Stolarsky \cite{sto} and Harborth \cite{ha}.
In particular, Harborth proved that $\alpha_2=1$ and derived the
numerical value $\gb_2\doteq0.812556$ to 6 decimal places; see
A077464 (Stolarsky-Harborth constant) for more information. Further
numerical estimates of $\gb_p$ for various $p$ have been made later;
of special mention is Chen and Ji's inequalities \cite{ch1}:
\begin{align}
	\frac1{(1+p^{-r})^{\rhop}}
	\min_{p^r\le n\le p^{r+1}}
	\frac{F_p(n)}{n^{\rhop}}
	\le \beta_p \le 
	\min_{p^r\le n\le p^{r+1}}
	\frac{F_p(n)}{n^{\rhop}},
\end{align}
which in principle makes it possible to calculate $\gb_p$ to any
given degree of precision. However, an exact expression remains 
unknown.

For $p\ge3$, Volodin \cite{vo2} conjectured that
\begin{equation}\label{beta3}
	\beta_3
	=\left(\frac{3}{2}\right)^{1-\varrho_3}
	=2^{\log_32-1},
\end{equation}
which was proved by Franco \cite{fr}; however, his proof does not 
extend to other primes $p$.

Wilson \cite{wi1} calculated $\beta_3,\beta_{5},\ldots,\beta_{19}$
to six decimal places and showed that
\begin{equation}\label{wilson0.5}
	\lim_{p\rightarrow\infty}\beta_p=0.5.
\end{equation}
He furthermore conjectured that
\begin{equation}
	\beta_{5}=\left(\frac{3}{2}\right)^{1-\varrho_{5}},
	\qquad\beta_{7}=\left(\frac{3}{2}\right)^{1-\varrho_{7}},
	\qquad\beta_{11}=\frac{59}{44}
	\left(\frac{22}{31}\right)^{\varrho_{11}}.
	\label{wc3}
\end{equation}

The main purpose of the present paper is to prove this conjecture,
and to give similar results for further primes $p$. More precisely,
by a detailed examination of the periodic function $\cP_p(x)$,
coupling with analytic bounds and numerical calculations, we are able
to find the minimum $\beta_p$ for all odd primes $3\le p\le 113$,
proving particularly Wilson's conjectures \eqref{wc3} and differently
\eqref{beta3}. Our approach can be readily extended to higher values
of $p$, but a proof for all odd primes $p$ remains open.

In our approach we fix an odd prime $p$.
After a change of variables (see \refS{Sour} for details), we obtain 
$\gb_p=\min_{s\in[p\qw,1]} G(s)$ for the function $G(s)=G_p(s)$ defined in
\eqref{maj2}. The main part of our argument is to show that (for the primes
$p$ that we have studied, at least) this minimum is attained at 
the point
\begin{align}\label{E:hat-sp}
	\hat{s}_p 
	=\hat{s}_p(\xi,\eta)
	:= \frac{2\xi+1}{2p}-\frac{\eta}{p^2},
\end{align}
for a suitable pair of integers $(\xi,\eta)$ with $1\le\xi<p$ and
$0\le\eta\le\frac{p-1}{2}$.
In terms of $p$-ary expansion, 
\begin{align}\label{E:hat-sp2}
\hat{s}_p = (0.b_1b_2\dots), 
\quad\text{where } b_1=\xi,\; 
b_2=\tfrac{p-1}2-\eta, \text{ and } b_j=\tfrac{p-1}2 \text{ for } j\ge3. 
\end{align}
When this holds, we thus have $\gb_p=B_{\xi,\eta}:=G(\hat{s}_p)$,
which explicitly is given by (see \refL{L:G-hat-s})
\begin{align}\label{E:Bml}
	B_{\xi,\eta} = 
	\binom{\xi+1}{2}\lrpar{1+
	\frac{(p-2\eta)(p-2\eta+1)}{2\xi p(p+1)}}
	\lrpar{\frac{2\xi+1}{2}
	- \frac{\eta}{p}}^{-\varrho_p}.
\end{align}
One complication is that the correct choice of $(\xi,\eta)$ is not obvious
and depends on $p$, as is illustrated in the following theorem, which is our
main result.

\begin{theorem}\label{T1}
Wilson's conjecture \eqref{wc3} holds true.
More generally, 
for an odd prime $p$, $3\le p\le113$, 
we have $\gb_p=B_{\xi,\eta}$
where the pair of values $(\xi,\eta)$ 
(together with $\hat{s}_p(\xi,\eta)$) are given in the following 
table:
\begin{align}\label{T1-3-0}
\renewcommand{\arraystretch}{1.5}
\begin{tabular}{c}
\begin{tabular}{ccccc} 
$p$ & $\{3,5,7\}$ & $\{11,13,17,19,23\}$ & $\{29\}$
& $\{31,37,41,43,47,53\}$ \\ \hline
$(\xi,\eta)$ & $(1,0)$ & $(1,1)$ & $(2,1)$ & $(2,2)$ \\
$\hat{s}_p(\xi,\eta)$ & $\frac{3}{2p}$ 
& $\frac{3}{2p}-\frac{1}{p^2}$ 
& $\frac{5}{2p}-\frac{1}{p^2}$
& $\frac{5}{2p}-\frac{2}{p^2}$ \\ \hline\hline
\end{tabular}\\
\begin{tabular}{ccccc}
$p$ & $\{59,61,67,71,73,79\}$
& $\{83,89,97,101,103,107\}$
& $\{109,113\}$ \\ \hline
$(\xi,\eta)$ & $(2,3)$ & $(2,4)$ & $(2,5)$ \\
$\hat{s}_p(\xi,\eta)$ 
& $\frac{5}{2p}-\frac{3}{p^2}$
& $\frac{5}{2p}-\frac{4}{p^2}$
& $\frac{5}{2p}-\frac{5}{p^2}$
\end{tabular}
\end{tabular}
\end{align}
\end{theorem}
Indeed, the same result $\beta_p = B_{\xi,\eta}$ holds for larger 
values of $p$ with suitably chosen $(\xi,\eta)$, and our numerical 
calculations confirmed this for $p$ up to several thousand (see 
Section~\ref{S:large-p}); however, a proof for \emph{all} odd primes 
remains open. 

In particular, we obtain 
\begin{align}
\renewcommand{\arraystretch}{2}
\begin{tabular}{ccccc}
$p$ & $13$ & $17$ &\quad & $\beta_p (p=3,\dots,113)$\\ \cline{1-3}
$\beta_p$ 
& \makecell{$\frac{124}{91}\bigl(\frac{26}{37}\bigr)^{\log_{13}91}$
\\ \footnotesize{$\approx 0.73266$}}
& \makecell{$\frac{71}{51}\bigl(\frac{34}{49}\bigr)^{\log_{17}153}$ 
\\ \footnotesize{$\approx 0.72758$}}
& &\multirow{2}{7em}{\includegraphics[height=3cm]{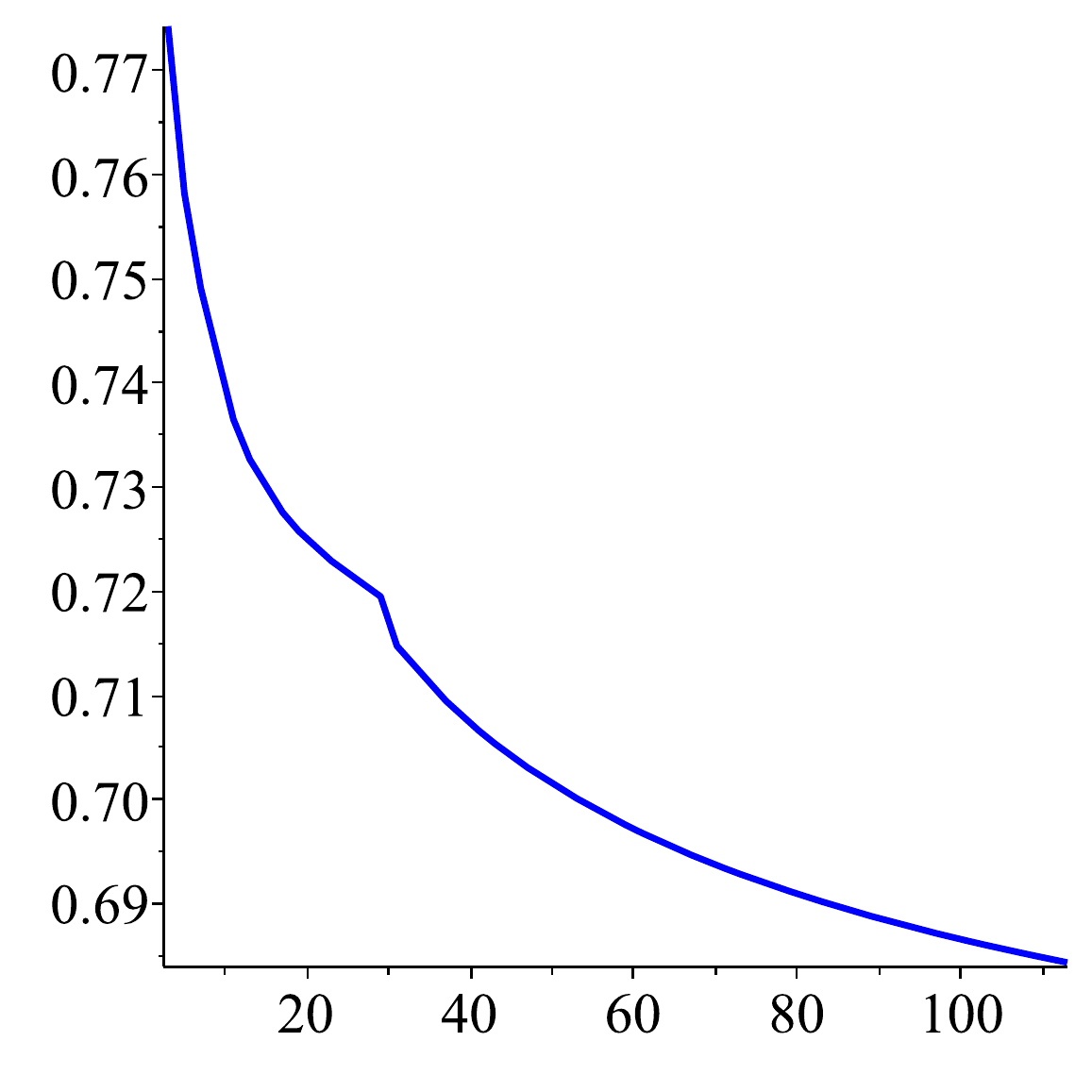}} \\
$p$ & $19$ & $113$ \\ \cline{1-3}
$\beta_p$ 
& \makecell{$\frac{533}{380}\bigl(\frac{38}{55}\bigr)^{\log_{19}190}$
\\ \footnotesize{$\approx 0.72575$}}
& \makecell{$\frac{7780}{2147}\bigl(\frac{226}{555}
\bigr)^{\log_{113}6441}$\\ \footnotesize{$\approx 0.68432$}}
\end{tabular} 
\end{align}

To prove \refT{T1}, in view of \eqref{gbp}, it suffices to find the
minimum of $\cP_p(\log_p s)$ for $s\in[p\qw,1]$ . Wilson (1998)
conjectured (in a different formulation for integers $n$, see
\refR{Rn} below) that for any $p$, the minimum $\beta_p$ occurs at
some point $\hs_p$ such that all but a finite number of its base $p$
digits are equal to $\frac{p-1}2$. Note that the point $\hat{s}_p$
of \eqref{E:hat-sp}--\eqref{E:hat-sp2} is of the type consistent with Wilson's
conjecture.

We
observe that the graph of $\cP_p(\log s)$ has a self-similar nature,
and if the graph is ``zoomed in" on such points $s$, the resulting 
function converges uniformly. We then find the local behavior from the
limiting function. The proof of \eqref{T1-3-0} then builds on this
idea, and this implies in particular Wilson's conjecture (\ref{wc3});
see \refS{Spf} for the details of the proof.

\begin{problem}\label{PR1}
The apparently simplest case $p=2$ seems to be actually the most
complicated. Despite many digits of $\beta_2$ are known (see A077464
and the references therein), to the best of our knowledge, no exact
expression for $\beta_2$ is available or proposed or conjectured.
How to characterize the minimum point $\hs_2$, and what is the 
corresponding minimum value $\gb_2$?
\end{problem}

\begin{problem}\label{PR2}
The main open question is whether $\beta_p = B_{\xi,\eta}$ 
for some $(\xi,\eta)$ for all 
odd primes $p$.
Equivalently, is the minimum point always some $\hat{s}_p$ 
(defined in \eqref{E:hat-sp})? 
\end{problem}

Our approach is based on the resolution of the recurrence
\eqref{rec2} satisfied by $F_p(n)$, which we prove below in \refT{T2},
following the same arguments used in our previous paper \cite{hk24}.
No other number-theoretic properties are needed. This then yields the
representation \eqref{Pp} with a continuous periodic functions
$\cP_p$, and a special explicit formula for $\cP_p$ that we will
use. Indeed, in the case $p=2$, this recurrence is of the binary form
studied in \cite{hk24}, and $F_2(n)$ was one of the many examples
discussed there. We show in Appendix \ref{Srec} that the method of
\cite{hk24} can be generalized to a general class of $p$-ary
recursions including \eqref{rec2}. (The results in the appendix are
valid for any integer $p\ge2$.)

More generally, a number of authors have studied $F_{p,d}(n)$, the
number of multinomial coefficients $\binom{m}{j_1,\dots,j_d}$ with
$0\le m<n$ that are not divisible by $p$. 
\begin{problem}\label{PR3}
  Extend the methods and results of the present paper to multinomial
  coefficients. 
\end{problem}

\section{A recurrence and its solutions}\label{Sbinom}

In this section, we fix a prime $p\ge2$.

\begin{theorem}\label{T2}
The total number of binomial coefficients $\binom{m}{k}$ with $m,k<n$ that
are not divisible by $p$ satisfies the recurrence
\begin{equation}\label{rec2}
	F_p(n)=\sum_{0\leq j<p}(p-j)F_p\left(\left\lfloor
    \frac{n+j}{p}\right\rfloor \right)\qquad (n\ge p),
\end{equation}
with the initial values $\{F_p(j)=\binom{j+1}{2}:j=1,\ldots,p-1\}$. 
In fact, \eqref{rec2} holds for all $n\ge0$, with $F_p(0):=0$.
\end{theorem}

\begin{proof}
It is obvious that $\binom{j}{k}$ is not divisible by a prime $p$ if
$k\leq j<p$. Thus the initial values are
\begin{equation}
	F_p(j)=\binom{j+1}{2}\quad\textrm{for }j=1,\ldots,p-1.
\end{equation}
We use the following expression from Volodin \cite{vo1} (there stated 
more generally for multinomial coefficients):
\begin{equation}
	F_p(n)
	=\frac{1}{2}\sum_{0\leq j\leq \nu}
	\binom{p+1}{2}^{j}b_{j}
	\prod_{j\leq i\leq \nu}(b_i+1),
	\label{vo2}
\end{equation}
[see also \cite{ch1}] where $b_{j}\in\set{0,\dots,p-1}$ are the base
$p$ digits of $n=b_0+b_1p+\dotsm+b_\nu p^\nu$. (The formula
\eqref{vo2} holds trivially for $n=0$ too, with an empty sum.) If
$n=kp$, then $b_{0}=0$ and $k=b_1+b_2p+\cdots+b_{\nu}p^{\nu-1}$.
Thus, from (\ref{vo2}),
\begin{equation}
	\begin{aligned}[b]
		F_p(n) 
		&=\frac{1}{2}\sum_{1\leq j\leq \nu}\binom{p+1}{2}^{j}b_{j}
		\prod_{j\leq i\leq \nu}(b_i+1)\\
		& =\binom{p+1}{2}\cdot\frac{1}{2}
		\sum_{1\leq j\leq \nu}\binom{p+1}{2}^{j-1}b_{j}
		\prod_{j\leq i\leq \nu}(b_i+1)\\
		& =\binom{p+1}{2}\cdot\frac{1}{2}
		\sum_{0\leq j\leq \nu-1}\binom{p+1}{2}^{j}b_{j+1}
		\prod_{j\leq i\leq \nu-1}(b_{i+1}+1)\\
		& =\binom{p+1}{2}F_p(k).
	\end{aligned}
\end{equation}
In general, suppose $n=kp+r$, where $k\ge0$ and $0\leq r<p$. Then 
$b_{0}=r$ and \eqref{vo2} yields
\begin{equation}
	F_p(n)-F_p(kp)
	=\frac{r}{2}\prod_{0\leq i\leq \nu}(b_i+1)
	=\frac{r(r+1)}{2}\prod_{1\leq i\leq \nu}(b_i+1)
\end{equation}
since all but the first term in the sums cancel. By Fine's result
(\ref{fine2}), this yields
\begin{equation}
	F_p(n)-F_p(kp)
	=\binom{r+1}{2}\left(F_p(k+1)-F_p(k)\right).
\end{equation}
Thus we have
\begin{equation}
	\begin{aligned}[b]
		F_p(n) 
		&= F_p(kp)+\left(F_p(n)-F_p(kp)\right)\\
		&= \binom{p+1}{2}F_p(k)+\binom{r+1}{2}
		\left(F_p(k+1)-F_p(k)\right)\\
		&= \sum_{0\leq j<p}(p-j)F_p(k)+\binom{r+1}{2}
		\left(F_p(k+1)-F_p(k)\right)\\
		&= \sum_{0\leq j\leq p-r-1}(p-j)F_p(k)
		+\sum_{p-r\leq j<p}(p-j)F_p(k+1)\\
		&= \sum_{0\leq j<p}(p-j)F_p\left(\left\lfloor
		\frac{n+j}{p}\right\rfloor \right).
	\end{aligned}
\end{equation}
This proves the recurrence \eqref{rec2}.
\end{proof}

\refT{T2} shows that $F_p(n)$ satisfies a recurrence of the type
treated in Appendix \ref{Srec}. From \refT{TP} we thus immediately
obtain the representation \eqref{Pp}, together with the following
formula for $\cP_p(t)$. (This formula is essentially given in
\cite{ch2}.) 

\begin{theorem}\label{T3} Define
\begin{align}\label{Ap}
	A=A_p:=\binom{p+1}{2}.   
\end{align}
Then the number of binomial coefficients in the first $n$ rows that 
are not divisible by $p$ satisfies 
\begin{equation}\label{t3a}
	F_p(n)
	=n^{\varrho}\mathcal{P}
	\left(\log_pn\right)\qquad\textrm{for all }n\geq1,
\end{equation}
where $\varrho=\rhop:=\log_pA$ and $\cP(t)=\cP_p(t)$ is a
continuous $1$-periodic function given by
\begin{equation}\label{t3b}
	\mathcal{P}(t):=A^{1-\{t\}}\varphi(p^{\{t\}-1}),
\end{equation}
with the function $\varphi=\gf_p:\oi\to\bbR$ given by the
explicit formula
\begin{equation}\label{t32}
	\varphi\lrpar{\sum_{j\geq1}b_{j}p^{-j}}
	=\frac{1}{2}\sum_{j\geq1}\frac{b_{j}}{A^{j}}
	\prod_{1\le i\le j}(b_i+1),
\end{equation}
for any $b_{j}\in\{0,1,2,\ldots,p-1\}$; furthermore, $\gf$ satisfies
that for $j=0,1,\ldots,p-1$,
\begin{equation}
	\varphi(t)
	=\frac{j+1}{A}\varphi(\{pt\})
	+\frac{\binom{j+1}{2}}{A},
	\qquad\mathrm{if}
	\quad\frac{j}{p}\leq t\leq\,\frac{j+1}{p}
	\label{t31}
.\end{equation}
\end{theorem}

\begin{proof}
By \refT{T2}, $F_p(n)$ satisfies the recurrence \eqref{rec1} with
$\gamma_i=p-i$. We have \begin{equation}\label{t33} \sum_{p-j\leq
i<p}\gamma_i =\binom{j+1}{2}=F_p(j)\quad\textrm{for
}j=1,\ldots,p-1, \end{equation} and thus the condition (\ref{d1}) is
satisfied; see also \refR{Rinitial}. Similarly, $A$ in \eqref{eq:A}
is given by \eqref{Ap}. The results follow by \refT{TP} and
\refL{L1}; \eqref{t31} follows by plugging $\gamma_{j}=p-j$ into
(\ref{l1}), and (\ref{E2}) yields
\begin{equation}
	\begin{aligned}[b]
	\varphi\left(\sum_{j\geq1}b_{j}p^{-j}\right) 
	&=\sum_{j\geq1}\frac{\sum_{i=p-b_{j}}^{p}(p-i)
	}{A^{j}}\prod_{1\le i<j}(b_i+1)\\
	& =\frac12\sum_{j\geq1}\frac{b_{j}(b_{j}+1)
	}{A^{j}}\prod_{1\le i <j}(b_i+1)\\
	& =\frac{1}{2}\sum_{j\geq1}\frac{b_{j}}{A^{j}}
	\prod_{1\le i\le j}(b_i+1),
	\end{aligned}
\end{equation}
for any $b_{j}\in\{0,1,2,\ldots,p-1\}$, which is \eqref{t32}.
\end{proof}

\section{Our approach}\label{Sour}

In this section, $p$ is a fixed odd prime. Recall that 
$A=\binom{p+1}2$ and $\varrho_p=\log_pA$.

By \eqref{gbp} and \refT{T3}, $\gb_p$ is the minimum of the periodic 
function
\begin{equation}\label{maj1}
	\mathcal{P}(t)
	:=A^{1-{t}}\varphi(p^{{t}-1})
	\qquad\textrm{for}\quad t\in[0,1),
\end{equation}
where $\gf$ is given by \eqref{t32}. (By continuity, we may as well 
take the minimum for $t\in[0,1]$.) We make the change of variables 
$s=p^{t-1}$ and consider
\begin{equation}\label{maj2}
	G(s)
	:=A^{-\log_ps}\varphi(s)
	\qquad\textrm{for}\quad s\in[p^{-1},1];
\end{equation}
thus $\mathcal{P}(t)=G(p^{t-1})$ for $t\in[0,1)$, and thus
$\gb_p = \min\set{G(s):s\in[p\qw,1]}$. Note that
\begin{align}\label{maj3}
    A^{-\log_p s} 
	= p^{-\rhop \log_p s}=s^{-\rhop}.
\end{align}

\begin{remark}\label{Rn}
For any $n\ge1$, by \eqref{t3a}--\eqref{t3b} and \eqref{maj1}--\eqref{maj2}, we see that 
\begin{align}
	F_p(n)n^{-\rhop}&
	=\cP(\log_p n)
	=\cP(\frax{\log_p n}) 
	=G\bigpar{p^{\frax{\log_p n}-1}}
	=G\bigpar{np^{-\floor{\log_pn}-1}}.
\end{align}
It follows that if $G$ attains its minimum on $[p\qw,1]$ at $\hs$,
then the sequence $n_k:=\floor{p^k \hs}$ satisfies
\begin{align}
	F_p(n_k)n_k^{-\rhop}
	\to G(\hs)
	=\gb_p,
\end{align}
and thus the infimum $\gb_p$ is asymptotically reached by the
sequence $(n_k)$. Conversely, Wilson \cite{wi1} conjectured (in a
somewhat stronger form) that
\begin{align}
\gb_p=\lim_{k\to\infty}
(F_p(n_k)n_k^{-\rhop}),  
\end{align}
for a sequence $n_k$ given by the recursion
$n_{k+1}=pn_k+ \frac{p-1}2$ for a suitably chosen $n_1$; this
sequence is of the form just mentioned (up to a shift of indices),
and Wilson's conjecture thus would imply that the minimum on
$[p\qw,1]$ is attained at a point $\hs$ such that all but a finite
number of the digits in base $p$ of $\hs$ are $\frac{p-1}2$. Note
that all points $\hs$ that we consider as potential minimum points
are of this type, see \eqref{E:hat-sp2}.
\end{remark}

Our methods of proof consists of two major techniques:
\emph{magnifying mapping} and \emph{piecewise monotonic
majorization}. The former defines first a mapping $\theta$ and
magnifies the local difference of $G(\theta(s))-G(\theta(\frac12))$
in a small neighborhood, say $J$, of $\hat s=\theta(\frac12)$ into
the global difference $\varphi(s)-\varphi(\frac12)$, justifying that
$G(\hat s)$ is a local minimum in $J$. The latter bounds crudely the
ratio between two monotonic functions by their extreme values in the
targeted interval, which, after partitioning the interval
$[0,1]\setminus J$ into proper subintervals, is used
interval-by-interval to check that $G(\hat s)$ is also a minimum in
$[0,1]\setminus J$.

\subsection{A magnifying mapping}

As the minimum $\hat s$ of $G(s)$ we are going to prove all have the
form \eqref{E:hat-sp} whose $p$-ary expansion has an infinity number
of trailing digits of the form $\frac{p-1}2$, we construct a linear
mapping as follows.

Fix $M$ and $m$ with $1\leq m<p^{M}$. Let $\mu$ be the middle 
point of $\bigl[\frac{m}{p^{M}},\frac{m+1}{p^{M}}\bigr]$: 
\begin{align}\label{E:mu-mid}
	\mu := \frac{m+\frac12}{p^{M}}.
\end{align}
For every $k\ge0$, define a linear mapping $\theta_k$ from $[0,1]$ 
onto the interval 
\begin{equation}\label{E:Ik}
	I_{M+k}
	:=\left[\mu-\frac{1}{2p^{M+k}}, 
	\mu+\frac{1}{2p^{M+k}}\right]
\end{equation}
by
\begin{equation}
	\theta_{k}(t)
	:=\frac{m}{p^{M}}
	+\sum_{M<j\leq M+k}\frac{p-1}{2\,p^{j}}
	+\frac t{p^{M+k}},\qquad t\in\oi.
	\label{t40}
\end{equation}
Thus $\mu=\theta_k(\frac12)$. In terms of the $p$-ary expansion, if 
\begin{align}\label{t45}
    m=a_1p^{M-1}+\cdots+a_{M-1}p+a_{M},
\end{align}
then $\mu$ has the form $\mu = (0.\hb_1\hb_2\cdots)_p$, where
\begin{equation}\label{t46}
	\hb_i
	=a_i\;\textrm{for}\;1\leq i\leq M
	\quad\textrm{and}\quad
	\hb_i
	=\frac{p-1}{2}\;\textrm{for}\; i\ge M+1.
\end{equation}
We prove that the ``zoomed" functions $G(\theta_k(t))-G(\mu)$, 
suitably scaled, converge uniformly on $\oi$, and we give a sufficient
condition for $G(\mu)$ to be a minimum in an explicitly specified 
interval.

Note that $\varphi(\mu)$ has the form,
by \eqref{t32} and \eqref{t46},
\begin{align}
	\varphi(\mu)
	&=\frac12\sum_{1\le j\le M}
	\frac{a_j\prod_{i=1}^j(1+a_i)}{A^j}
	+\frac{\tau_M^{}}{4}
	= \frac12\sum_{1\le j\le M}
	\frac{a_j\prod_{i=1}^j(1+a_i)}{A^j}
	+\tau_M^{}\varphi\lrpar{\frac12},
\end{align}
where
\begin{align}\label{tau}
	\tau_M^{} :=\frac{\prod_{i=1}^M(1+a_i)}{A^M}.
\end{align}

The construction of the mapping $\theta_k$ is helpful in bringing the
local difference $\varphi(\theta_{k}(t))
-\varphi(\theta_{k}(\frac12))$ into a global one in terms of
$\varphi(t)-\varphi(\frac12)$.
\begin{lemma}\label{L4}
We have 
\begin{equation}\label{l4}
	p^{k}\bigpar{\varphi(\theta_{k}(t))-\varphi(\mu)}
	=\tau_M^{}\left(\varphi(t)
	-\varphi\left(\frac{1}{2}\right)\right).
\end{equation}  
\end{lemma}

\begin{proof}
Let $b_i$ and $\tilde{b}_i$ be the base $p$ digits of $t$ and
$\theta_{k}(t)$, respectively. Then \eqref{t40} shows that the first
$M+k$ digits $\tb_i$ coincide with those of $\mu$ given by 
\eqref{t46}:
\begin{equation}
	\tilde{b}_i
	=\hb_i=a_i\;\textrm{for}\;1\leq i\leq M
	\quad\textrm{and}\quad
	\tilde{b}_i=\hb_i=\frac{p-1}{2}\;
	\textrm{for}\;M+1\leq i\leq M+k,
\end{equation}
and also that the remaining digits of $\tb_i$ are the digits of $t$,
i.e.,
\begin{equation}\label{l4b}
	\tilde{b}_{i+k+M}=b_i,
\qquad i\ge1.
\end{equation}
Hence, if we compute $\gf(\theta_k(t))$ and $\gf(\mu)$ by
\eqref{t32}, then the first $M+k$ terms are equal, and we obtain
\begin{align}\label{l4c}
	\varphi(\theta_{k}(t))-\varphi(\mu) 
	&=\frac{1}{2}\sum_{j\geq k+M+1}\frac{\tilde{b}_{j}
	\prod_{i=1}^{j}\left(\tilde{b}_i+1\right)}
	{A^{j}}\\
	&\qquad  
	-\frac{1}{2}\sum_{j\geq k+M+1}\frac{\frac{p-1}{2}
	\prod_{i=1}^{M}\left({a}_i+1\right)
	\prod_{i=M+1}^{j}\left(\frac{p-1}{2}+1\right)}{A^{j}}
	\notag\\
	&=\frac{\prod_{i=1}^{M}\left({a}_i+1\right)}
	{A^{k+M}}\left(\frac{p+1}{2}\right)^{k}
	\frac{1}{2}\left(\sum_{l\geq1}
	\frac{b_{l}\prod_{i=1}^{l}\left(b_i+1\right)}
	{A^{l}}-\sum_{l\geq1}\frac{\frac{p-1}{2}
	\left(\frac{p+1}{2}\right)^{l}}{A^{l}}\right)
	\notag\\
	&=\frac{\tau_M^{}}{p^{k}}\left(\varphi(t)
	-\varphi\left(\frac{1}{2}\right)\right).\qedhere
\end{align}
\end{proof}

The crucial properties we need of the magnifying mapping $\theta_k$ 
are given as follows, the first for large $k$ and the second for 
finite one. 
\begin{theorem}\label{T4}
With the notations as above, we have, uniformly for $t\in\oi$,
\begin{equation}
	p^k\left(G(\theta_{k}(t))-G(\mu)\right)
	= \mu^{-\rhop-1}
	\lrpar{\mathcal{Q}_{\mu}(t)+O\left(p^{-k}\right)},
	\label{t41}
\end{equation}
for large $k$, where the limiting function $\mathcal{Q}_{\mu}(t)$ is 
given by 
\begin{equation}
	\mathcal{Q}_{\mu}(t)
	:=\tau_M^{}\mu
	\left(\varphi(t)
	-\varphi\left(\frac{1}{2}\right)\right)
	 - \frac{\rhop \varphi(\mu)}{p^M}
	\left(t-\frac{1}{2}\right).
	\label{t42}
\end{equation}
Furthermore, if for some $k\ge0$,
\begin{equation}
	\mathcal{Q}_{\mu}(t)
	\ge E_{\mu,k}(t)
	\qquad\text{for all}
	\quad t\in\left[0,\frac{1}{2}-\frac{1}{2p}\right]
	\cup\left[\frac{1}{2}+\frac{1}{2p},1\right],
	\label{t43}
\end{equation}
where
\begin{equation}
	E_{\mu,k}(t)
	:=\frac{\tau_M^{}\rhop}{p^{M+k}}
	\left(\varphi(t)
	-\varphi\left(\frac{1}{2}\right)\right)
	\left(t-\frac{1}{2}\right),	
	\label{t44}
\end{equation}
then $G(\mu)$ is the minimum of the function $G$ in the interval
$I_{M+k}$ (defined in \eqref{E:Ik}), and this minimum is attained 
only at $\mu$.
\end{theorem}
\begin{proof}
For simplicity, we use the abbreviations (with an abuse of notation):
$\nabla G(\theta) :=G(\theta_k(t))-G(\mu)$, $\nabla\varphi(\theta) :=
\varphi(\theta_k(t))-\varphi(\mu)$, $\nabla \varphi :=
\varphi(t)-\varphi(\frac12)$, $\nabla \theta := \theta_k(t)-\mu$ and
$\nabla t := t-\frac12$. Observe first that a Taylor expansion yields
\begin{align}
	\theta^{-\rhop}-\mu^{-\rhop} 
	+ \rhop\mu^{-\rhop-1}(\theta-\mu)
	= J_{\theta,\mu}\mu^{-\rhop-1} (\theta-\mu)^2,
\end{align}
where 
\begin{align}\label{E:J}
	J_{\theta,\mu} := \mu^{\rhop+1}
	\rhop(\rhop+1)\int_0^1 
	x(\mu x+\theta(1-x))^{-\rhop-2}\text{d}x
\end{align}
remains positive whenever $\theta,\mu>0$ and $\theta\ne\mu$ 
(or $t\ne\frac12$). Applying this expansion, we obtain
\begin{align}
	\nabla G(\theta) 
	&= \theta_k(t)^{-\rhop}
	\bigl(\varphi(\theta_k(t))-\varphi(\mu)\bigr)
	+\varphi(\mu)\bigl(\theta_k(t)^{-\rhop}
	-\mu^{-\rhop} \bigr)\notag\\
	&= \mu^{-\rhop-1}\nabla\varphi(\theta)
	\lrpar{\mu-\rhop\nabla\theta
	+ J_{\theta_k(t),\mu}  
	\bigl(\nabla\theta\bigr)^2}\notag\\
	&\qquad - \varphi(\mu)\mu^{-\rhop-1}
	\lrpar{\rhop\nabla\theta
	-J_{\theta_k(t),\mu} \bigl(\nabla\theta \bigr)^2};
\end{align}
thus
\begin{equation}
	\begin{split}
		\mu^{\rhop+1}\nabla G(\theta) 
		&= \mu \nabla\varphi(\theta)  
		- \rhop\varphi(\mu)\nabla\theta
		- \rhop\nabla\varphi(\theta) \nabla\theta
		+J_{\theta_k(t),\mu}\varphi(\theta_k(t))
		\bigl(\nabla\theta \bigr)^2.\label{E:mb2}
	\end{split}
\end{equation}
By \eqref{t40}
\begin{equation}\label{E:mb3}
	p^{k}\nabla\theta
	=p^{-M}\nabla t.
\end{equation}
This, together with \eqref{E:mb2}, \refL{L4}, and the definitions  
\eqref{t42} and \eqref{t44} of $\mathcal{Q}_\mu$ and $E_{\mu,k}$, 
gives
\begin{equation}\label{mb4}
	p^{k}\nabla G
	=\mu^{-\rhop-1}(\mathcal{Q}_{\mu}(t)
	-E_{\mu,k}(t)+R_{\mu,k}(t)),
\end{equation}
where 
\begin{align}\label{mb4.5}
	R_{\mu,k}(t) := p^kJ_{\theta_k(t),\mu}\varphi(\theta_k(t))
	\bigl(\nabla\theta \bigr)^2.
\end{align}
We note for later use that, by \eqref{mb4.5} and \eqref{E:J},
\begin{align}\label{mb4.6}
  	R_{\mu,k}(t)>0,\qquad\text{if } \theta_k(t)\neq\mu
\quad(\text{i.e., } t\neq\tfrac12).
\end{align}
Since $p$, $M$, and $(a_i)_1^M$ are fixed, and $\gf(t)$ 
is bounded, we see that 
\begin{align}\label{mb5}
	R_{\mu,k}(t) 
	= O\lrpar{p^{-k}},
\end{align}
Similarly, \eqref{t44} implies that 
\begin{equation}\label{mb6}
	E_{\mu,k}(t)
	=O\left(p^{-k}\right).
\end{equation}
Hence, \eqref{t41} follows from \eqref{mb4}, \eqref{mb5}, and \eqref{mb6}. 

Now assume that (\ref{t43}) holds for some $k$. Then it also holds
for all larger $k$ as well, since $E_{\mu,k}(t)\ge0$ and the only
factor in \eqref{t44} that depends on $k$ is $p^{-k}$.

Let $x\in I_{M+k} =\left[\mu-\frac{1}{2}p^{-(M+k)},
\mu+\frac{1}{2}p^{-(M+k)}\right]$ with $x\neq \mu$. Let $k_{x}\ge k$
be the largest integer such that $x\in I_{M+k_x}$, and let $t_x\in\oi$
be such that $\theta_{k_x}(t_x)=x$. Then $x\notin I_{M+k_x+1}$ and 
thus
\begin{equation}
	t_{x}\in\left[0,\frac{1}{2}-\frac{1}{2p}\right)
	\cup\left(\frac{1}{2}+\frac{1}{2p},1\right].
\end{equation}
Hence, by \eqref{mb4}, \eqref{mb4.6}, and \eqref{t43},
\begin{equation}
	\begin{split}
    \mu^{\rhop+1}	
	p^{k_{x}}\left(G(x)-G(\mu)\right)
	&=\mu^{\rhop+1}
p^{k_{x}}\bigpar{G(\theta_{k_{x}}(t_{x}))-G(\mu)}\\
	&=\mathcal{Q}_{\mu}(t_x)-E_{\mu,k_x}(t_x)+R_{\mu,k_x}(t_x) \\
	&> \mathcal{Q}_{\mu}(t_x)-E_{\mu,k_x}(t_x) 
	\ge0.
	\end{split}
\end{equation}
Thus $G(x)>G(\mu)$ for every $x\neq\mu$ in $I_{M+k}$, which shows that
$\mu$ is the unique minimum point of $G$ in the interval $I_{M+k}$.
\end{proof}

\subsection{Monotonic majorization}

Once we convert the minimality of $G(s)$ at $s=\mu$ in $I_{M+k}$ to
the positivity of $\Delta_{\mu,k}(t) :=\mathcal{Q}_{\mu}(t)-
E_{\mu,k}(t)$ for $t$ in the unit interval excluding a small
neighborhood of $t=\frac12$ 
(see \eqref{t43}),
we then need means of handling the
positivity of the difference of two monotonic functions because
$\Delta_{\mu,k}(t)$ can be expressed as:
\begin{equation}\label{E:Delta}
	\begin{split}
	\Delta_{\mu,k}(t)
	&:= \mathcal{Q}_{\mu}(t)-E_{\mu,k}(t) \\
	&= \begin{cases}
		\frac{\rhop \varphi(\mu)}{p^M}
		\lrpar{\frac12-t} - \tau_M
		\lrpar{\varphi\left(\frac{1}{2}\right)-\varphi(t)}
		\lrpar{\mu+\frac{\rhop}{p^{M+k}}
	    \lrpar{\frac12-t}}, &\text{if }t\in[0,\tfrac12],\\
	    \tau_M\mu\lrpar{
	    \varphi(t)-\varphi\left(\frac{1}{2}\right)}
		-\frac{\rhop}{p^M}\lrpar{t-\frac12}
		\lrpar{\varphi(\mu)+
		\frac{\tau_M}{p^{k}}
		\lrpar{\varphi(t)-\varphi\left(\frac{1}{2}\right)}},
    	&\text{if }t\in[\frac12,1],
	\end{cases}
	\end{split}
\end{equation}
the first being the difference of two positive decreasing functions, 
the second that of two increasing functions, and the condition 
\eqref{t43} then being equivalent to $\Delta_{\mu,k}(t)\ge0$.

On the other hand, $G(s)=s^{-\rhop}\varphi(s)$ can also be regarded
as the ratio of two increasing functions. Thus to show that $G$
attains nowhere the minimum value $\beta$ for $s$ outside a small
neighborhood of $\mu$, we need to handle the ratio of two increasing
functions.

Since the function $\varphi(t)$ is of fractal type, we use the 
following simple idea.
\begin{lemma} \label{L3}
Assume that $f,g$ are increasing functions on $[a,b]$ with 
$f\ge0$ and 
$g>0$ 
there. If $f(a)/g(b)>C$, then $f(x)/g(x)>C$ for $x\in[a,b]$. 
Similarly, if $f(a)-g(b)>C$, then $f(x)-g(x)>C$ for $x\in[a,b]$.
\end{lemma}
\begin{proof}
By monotonicity, for $x\in[a,b]$,
\begin{align}
	\frac{f(x)}{g(x)}\ge\frac{f(a)}{g(b)}>C. 
\end{align}
The difference version is similar: $f(x)-g(x)\geq f(a)-g(b)>C$.
\end{proof}
The case when $f, g$ are decreasing functions is similar.

In particular, for $\delta>0$, if $(x+\delta)^{-\rhop}\varphi(x)
>\beta$, then $G(t)=t^{-\rhop}\varphi(t)>\beta$ for $t\in
[x,x+\delta]$.

Such a simple idea will be applied numerically to sufficiently small 
subintervals after a suitable partition of the target interval.

\section{Proof of \refT{T1} for $p=3,5,7$}\label{Spf} 

We begin with the proof of \eqref{E:Bml}. Again $p\ge3$ is a prime 
number.
\begin{lemma} \label{L:G-hat-s}
$G(\hat{s}_p)=B_{\xi,\eta}$ given in \eqref{E:Bml}, 
where $\hat{s}_p=\frac{2\xi+1}{2p} -\frac{\eta}{p^2}$. 
\end{lemma}
\begin{proof}
The $p$-ary expansion of $\hat{s}_p$ has the form $\hat{s}_p= 
(0.b_1b_2\dots)_p$ with $b_1=\xi$, $b_2=\frac{p-1}2-\eta$ and $b_j = 
\frac{p-1}2$ for $j\ge3$. Then, by \eqref{t32}, 
\begin{align}
	\frac{2\varphi(\hat{s}_p)}{\xi+1}
	&= \frac{\xi}{A}
	+ \frac{1}{A^2}\lrpar{\frac{p-1}2-\eta}
	\lrpar{\frac{p+1}2-\eta}\notag\\
	&\qquad+\frac{p-1}2
	\lrpar{\frac{p+1}2-\eta}
	\sum_{k\ge3}\frac1{A^k}
	\lrpar{\frac{p+1}2}^{k-2}\notag\\	
	&= \frac{\xi}{A}+\frac{(p-2\eta)(p-2\eta+1)}{4A^2},
\end{align}
or
\begin{align}\label{E:phi-hats}
	\varphi(\hat{s}_p)
	= \frac{\xi+1}{2A}\lrpar{\xi+
	\frac{(p-2\eta)(p-2\eta+1)}{2p(p+1)}}.
\end{align}
Thus 
\begin{align}
	G(\hat{s}_p)
	= \hat{s}_p^{-\rhop}\varphi(\hat{s}_p)
	= \frac{\xi+1}{2}\lrpar{\xi+
	\frac{(p-2\eta)(p-2\eta+1)}{2p(p+1)}}
	\lrpar{\frac{2\xi+1}{2}
	- \frac{\eta}{p}}^{-\varrho_p},
\end{align}
which is the same as \eqref{E:Bml}.
\end{proof}

We prove Wilson's conjecture for $p=3,5,7$, and establish the values
$\gb_3$, $\gb_5$, $\gb_7$ in \eqref{beta3} and \eqref{wc3}. Let
$s_1:=\hat{s}_p(1,0)=\frac{3}{2p}$, whose $p$-ary expansion is of the
form $(0.1bbb\dots)_p$, where $b=\frac{p-1}{2}$.

\subsection{Minimality of $G(s)$ in $I_2=[\frac{3}{2p}-\frac1{2p^2},
\frac{3}{2p}+\frac1{2p^2}]$}

Take $m=M=1$ and $k=1$ in Theorem~\ref{T4} so that 
$\mu=s_1=\frac{3}{2p}$ and $I_2 = [\frac{3}{2p}-\frac1{2p^2},
\frac{3}{2p}+\frac1{2p^2}]$. 
We have $a_1=1$ and we obtain from \eqref{tau}
and \eqref{E:phi-hats} with $(\xi,\eta)=(1,0)$, 
\begin{align}
  \tau_1=\frac{2}{A}
\qquad\text{and}\qquad
	\varphi(\mu)=\frac{3}{2A}.
\end{align}
Then \eqref{t42} and \eqref{t44} in \refT{T4} yield
\begin{align}
	\mathcal{Q}_{s_1}(t)
	&=\frac{3}{pA}\biggl(\underbrace{\varphi(t)
	-\varphi\left(\frac{1}{2}\right)}_{=:K_1(t)}
	\underbrace{-\frac{\rhop}{2}
	\left(t-\frac{1}{2}\right)}_{=:K_2(t)}\biggr),
\end{align}
and
\begin{equation}\label{mc2}
	E_{s_1,1}(t)
	=\frac{2\rhop}{Ap^2}\left(\varphi(t)
	-\varphi\left(\frac{1}{2}\right)\right)
	\left(t-\frac{1}{2}\right)
	=: \frac{3}{pA} K_3(t).
\end{equation}
Thus, we can write
\begin{align}
	\frac{pA}{3}\,\Delta_{\mu,1}(t)
	= \begin{cases}
	    K_2(t) - (-K_1(t)+K_3(t)),
		    &\text{if }t\in[0,\frac12],\\
		K_1(t) - (-K_2(t)+K_3(t)), 
		    &\text{if }t\in[\frac12,1],
	\end{cases}
\end{align}
where $K_2(t)$ and $-K_1(t)+K_3(t)$ are both positive and decreasing 
for $t\in[0,\frac12]$, and $K_1(t)$ and $-K_2(t)+K_3(t)$ are both 
positive and increasing for $t\in[\frac12,1]$. According to 
\refT{T4}, if 
\begin{align}\label{E:QE-357}
	\Delta_{\mu,1}(t)\ge0 \eqtext{for}
	t\in\left[0,\frac12-\frac1{2p}\right]
	\cup \left[\frac12+\frac1{2p},1\right],
\end{align}
then $s_1$ is the minimum of $G(s)$ for $s\in[\frac12-\frac1{2p^2},
\frac12+\frac1{2p^2}]$. To check the validity of \eqref{E:QE-357}, we
partition the two intervals in \eqref{E:QE-357} into
equally-spaced subintervals in each of which we apply the idea used in
Lemma~\ref{L3}; namely,  
for some (large) integers $N_1$ and $N_2$,
\begin{align}
	\begin{cases}
		K_2(t_{j+1})
		- \lrpar{-K_1(t_j)
		+ K_3(t_j)}\ge0,
		&\text{with }t_j=\frac{(p-1)j}{2pN_1},\\
		K_1(\tilde{t}_{j})
		- \lrpar{-K_2(\tilde{t}_{j+1})
		+ K_3(\tilde{t}_{j+1})}\ge0,
		&\text{with }\tilde{t}_j=
		\frac12+\frac1{2p}+\frac{(p-1)j}{2pN_2},
	\end{cases}
\end{align}
for $j=0,1,\dots,N_i-1$, $i=1,2$. This process is purely numerical 
and brings the condition \eqref{E:QE-357} into a finitely computable 
one.

For example, take $p=3$. Then $N_1=9$ is sufficient for 
$t\in[0,\frac12-\frac1{2p}]$, and $N_2=7$ for 
$t\in[\frac12+\frac1{2p},1]$. More precisely, we have the numerical 
values in each case:
\begin{center}
\renewcommand{\arraystretch}{1.5} 	
\begin{tabular}{c|ccccccccc}
	\multicolumn{10}{c}{$\Delta_{\mu,1}(t)> K_2(t_{j+1})
		- \lrpar{-K_1(t_j)+ K_3(t_j)}$, 
	$t_j=\frac{(p-1)j}{2pN_1}$ \& $N_1=9$}\\ \hline
	$j$ & $1$ & $2$ & $3$ & $4$ & $5$ & $6$ & $7$ & $8$ & $9$ \\ 
	$\Delta_{\mu,1}>$ & $0.082$ & $0.060$ & $0.044$ & $0.033$ 
	& $0.016$ & $0.009$ & $0.012$ & $0.00001$ & $0.001$ 
\end{tabular}

\medskip

\begin{tabular}{c|ccccccccc}
	\multicolumn{8}{c}{
	\makecell{$\Delta_{\mu,1}(t)> K_1(\tilde{t}_{j})
		- \lrpar{-K_2(\tilde{t}_{j+1})
		+ K_3(\tilde{t}_{j+1})}$\\ 
	$\tilde t_j=\frac12+\frac1{2p}+\frac{(p-1)j}{2pN_2}$ 
	\& $N_2=7$}}\\ \hline
	$j$ & $1$ & $2$ & $3$ & $4$ & $5$ & $6$ & $7$ \\ 
	$\Delta_{\mu,1}>$ & $0.054$ & $0.023$ & $0.013$ & $0.006$ 
	& $0.016$ & $0.043$ & $0.023$ 
\end{tabular}
\end{center}
Similarly, for $p=5$, we use $(N_1,N_2)=(35,10)$, and for $p=7$, 
$(N_1,N_2)=(114,17)$, respectively. 

In this way, we prove, by Theorem~\ref{T4}, that $s=s_1=\frac3{2p}$
is the minimum of $G(s)$ for $s\in I_2 =[s_1-\frac1{2p^2}, 
s_1+\frac1{2p^2}]$.

\subsection{$G(s)$ in $[\frac1p,1]\setminus I_2$}

Following the same numerical procedure used above for justifying the
minimality of $G(s)$ at $s=s_1$ for $s\in I_2$, we partition the two
intervals $[\frac1p, \frac3{2p}-\frac1{2p^2}]$ and
$[\frac3{2p}+\frac1{2p^2}, 1]$ into $N_3$ and $N_4$ subintervals, and
check, by the simple monotonic bounds in Lemma~\ref{L3}, that
$G(s)>\beta$ for $s$ in each of these subintervals.

Since $G(s) = s^{-\rhop}\varphi(s)$ is the ratio of two increasing 
functions in the unit interval, we check the conditions 
\begin{equation}\label{E:Gsb-357}
	\begin{split}
	&G(s)-\beta \\
	&\quad> \left\{\begin{array}{lll}
		\sigma_{j+1}^{-\rhop}\varphi(\sigma_j)-\beta>0
		,& \text{if }\sigma_j 
		:= \frac1p+\frac{(p-1)j}{2p^2N_3},
		& j=0,\dots,N_3-1\\
		\tilde\sigma_{j+1}^{-\rhop}
		\varphi(\tilde\sigma_{j})-\beta>0
		,& \text{if }\tilde\sigma_j 
		:= \frac3{2p}+\frac1{2p^2}+
		\frac{(2p^2-3p-1)j}{2p^2N_4},&
		j=0,1,\dots,N_4-1,
	\end{array}\right.
	\end{split}
\end{equation}
for $s$ in each of the subintervals $[\sigma_j, \sigma_{j+1}]$ and 
$[\tilde\sigma_j, \tilde\sigma_{j+1}]$, respectively.

For example, for $p=3$, we can take $N_3=9$ and $N_4=16$ for which 
the corresponding numerical values are listed as follows. 

\begin{center}
\begin{tabular}{c|ccccccccc}
	\multicolumn{10}{c}{$G(s)-\beta>
	\sigma_{j+1}^{-\rhop}\varphi(\sigma_j)
	-\beta>0$, $\sigma_j$ given in \eqref{E:Gsb-357} 
	\& $N_3=9$}\\ \hline
	$j$ & $1$ & $2$ & $3$ & $4$ & $5$ & $6$ & $7$ & $8$ & $9$ \\ 
	\makecell{$\sigma_{j+1}^{-\rhop}\varphi(\sigma_j)
	-\beta$}
	& $0.168$ & $0.123$ & $0.091$ & $0.067$ 
	& $0.040$ & $0.027$ & $0.023$ & $0.008$ & $0.008$
\end{tabular}	

\medskip

\begin{tabular}{c|cccccccc}
	\multicolumn{9}{c}{$G(s)-\beta>
	\tilde\sigma_{j+1}^{-\rhop}\varphi(\tilde\sigma_{j})
	-\beta>0$, $\tilde\sigma_j$ given in \eqref{E:Gsb-357} 
	\& $N_4=16$}\\ \hline
	$j$ & $1$ & $2$ & $3$ & $4$ & $5$ & $6$ & $7$ & $8$ \\ 
	\makecell{$\tilde\sigma_{j+1}^{-\rhop}
	\varphi(\tilde\sigma_{j})-\beta$)}
	& $0.029$ & $0.002$ & $0.006$ & $0.044$ & $0.131$ 
	& $0.090$ & $0.060$ & $0.049$ \\ \hline
	$j$ & $9$ & $10$ & $11$ & $12$ & $13$ & $14$ & $15$
	& $16$ \\
	$\tilde\sigma_{j+1}^{-\rhop}
	\varphi(\tilde\sigma_{j})-\beta$
	& $0.054$ & $0.033$ & $0.026$ & $0.042$ & $0.088$ 
	& $0.076$ & $0.079$ & $0.112$
\end{tabular}	
\end{center}

Similarly, for $p=5$, we can take $(N_3,N_4)=(35, 77)$, and for
$p=7$, $(N_3,N_4)=(147, 214)$. 
This completes the proof that $G(s)$ attains its minimum 
in $[\frac1p,1]$
at $s_1$ for $p=3,5,7$;
consequently $\gb_p=G(s_1)$ which yields 
the values $\gb_3$, $\gb_5$, $\gb_7$ in
\eqref{beta3} and \eqref{wc3} and
proves 
\refT{T1} (and Wilson's conjecture) for these $p$.
 
\begin{figure}
\begin{center}	
	\begin{tabular}{ccc}
		\includegraphics[scale=0.2]{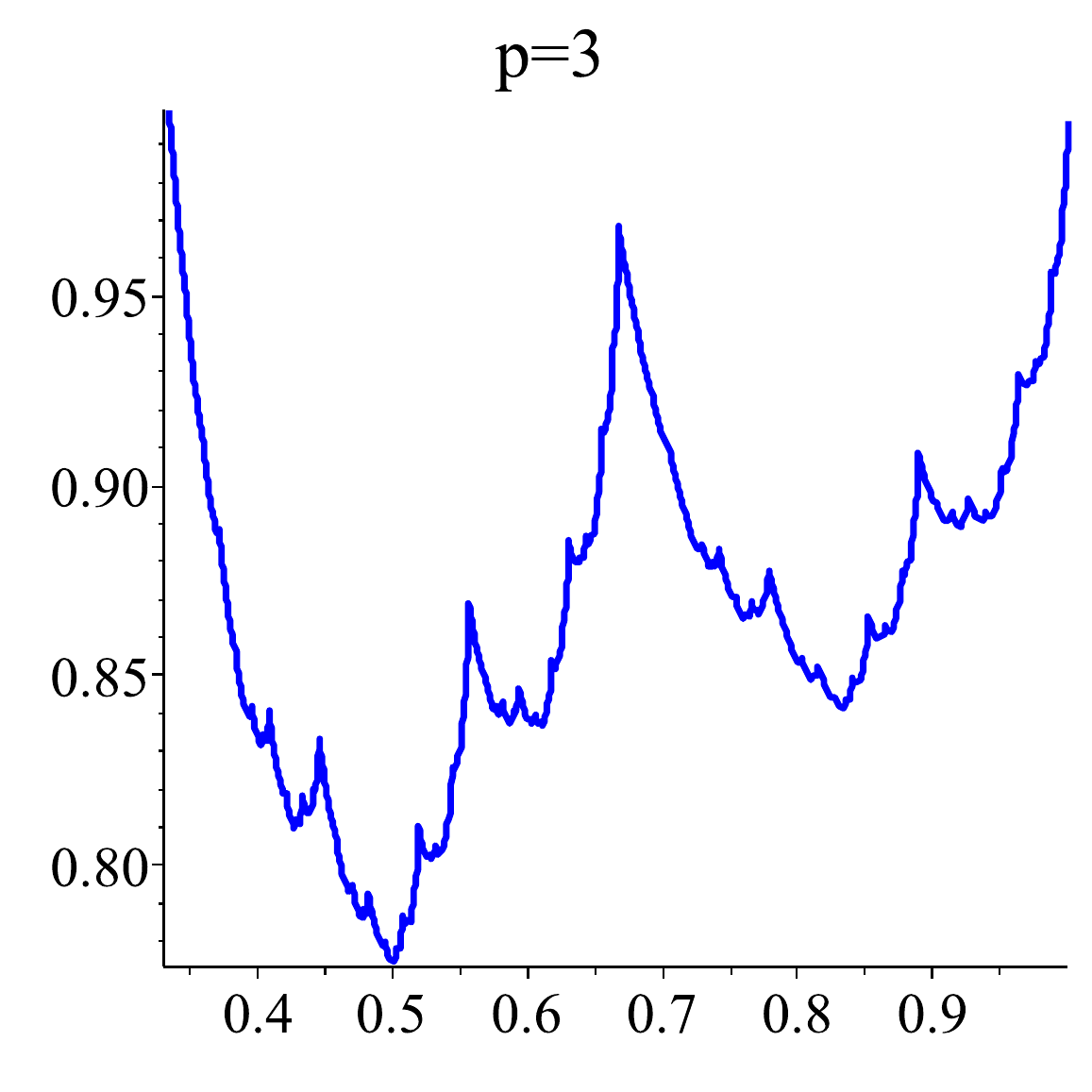}
		& \includegraphics[scale=0.2]{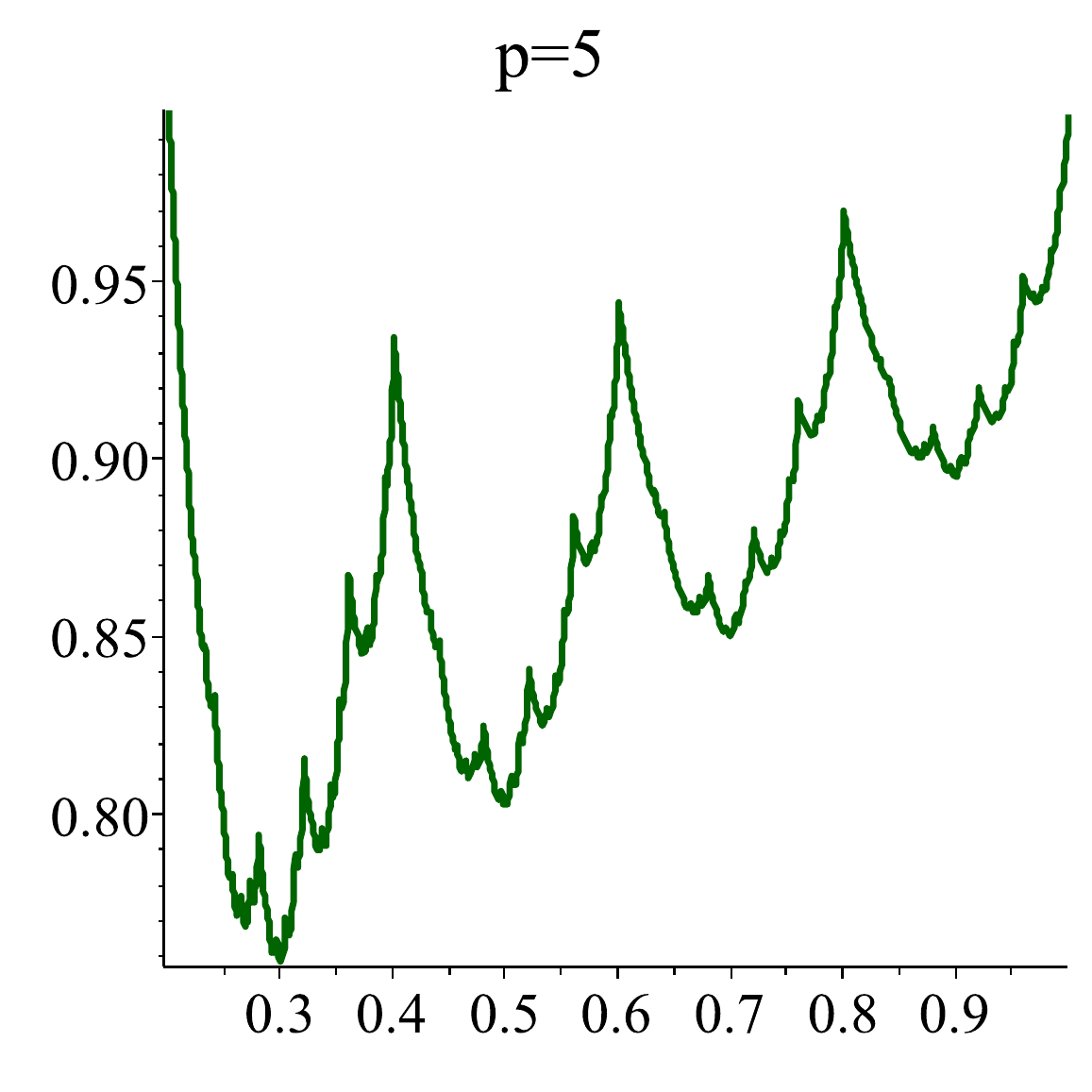}
		& \includegraphics[scale=0.2]{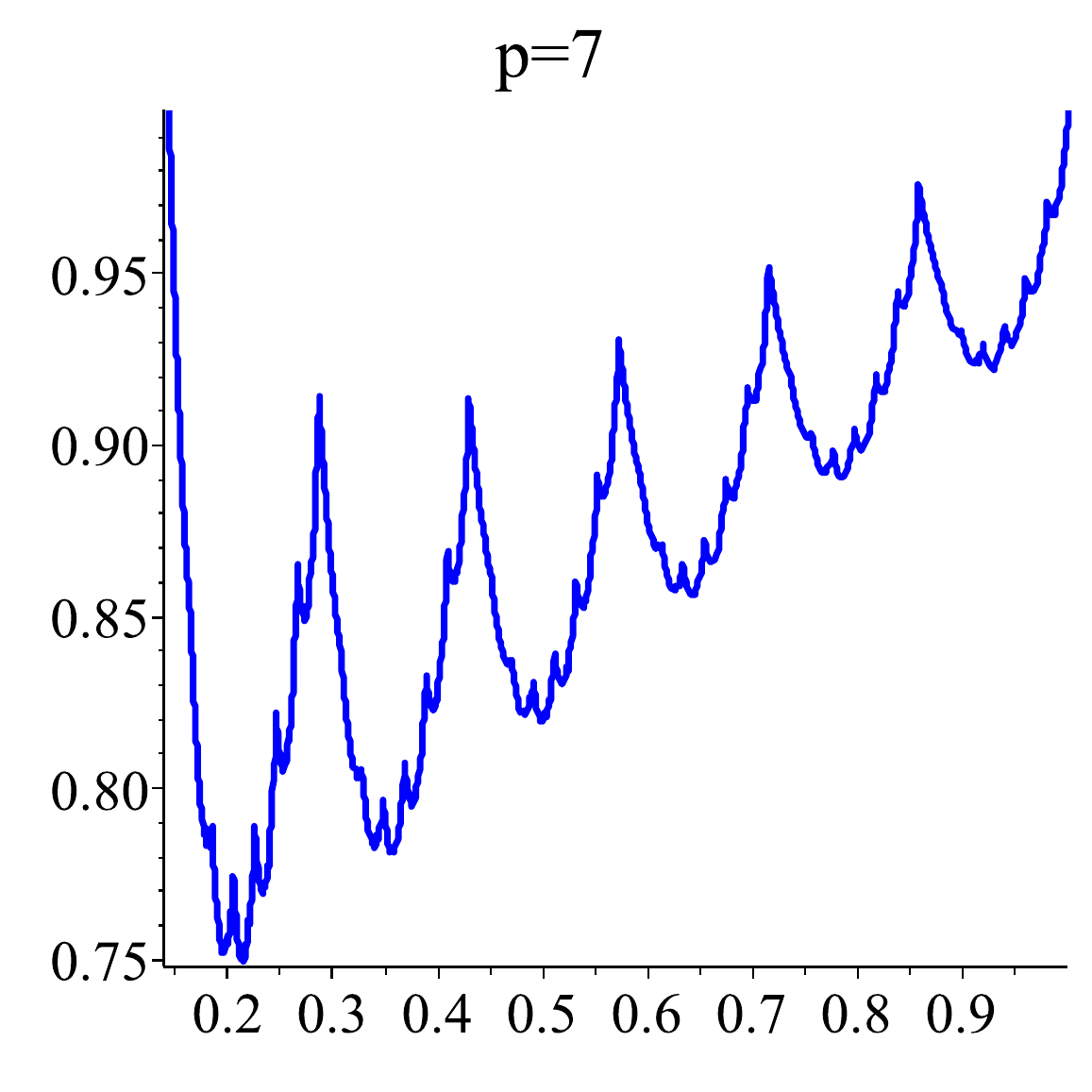}
	\end{tabular}
\end{center}
\caption{Fluctuations of $G(s)$, $s\in[p\qw,1]$, for $p=3,5,7$.}
\label{G357}
\end{figure}

\begin{figure}
\begin{center}	
	\begin{tabular}{ccc}
		\includegraphics[scale=0.2]{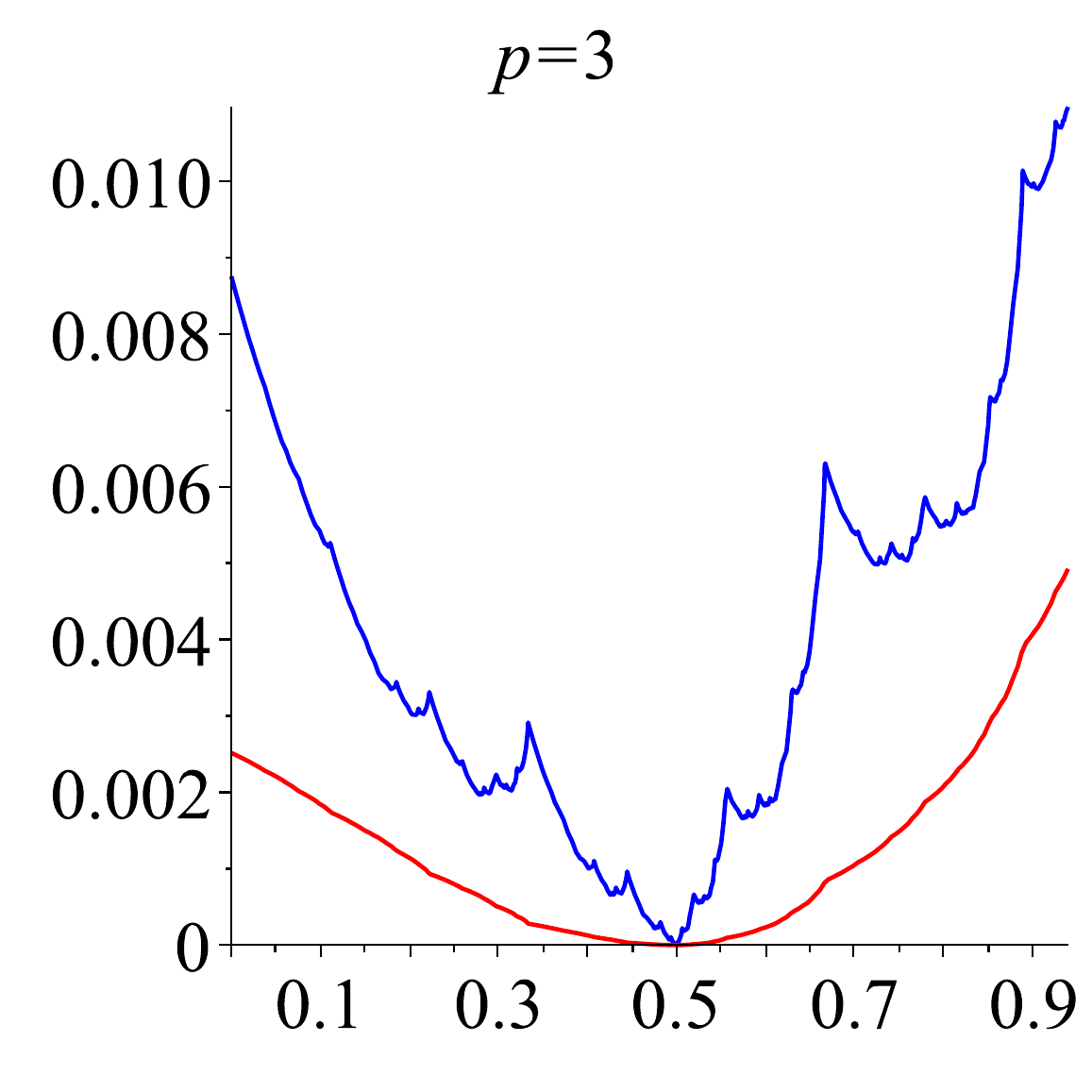}
		& \includegraphics[scale=0.2]{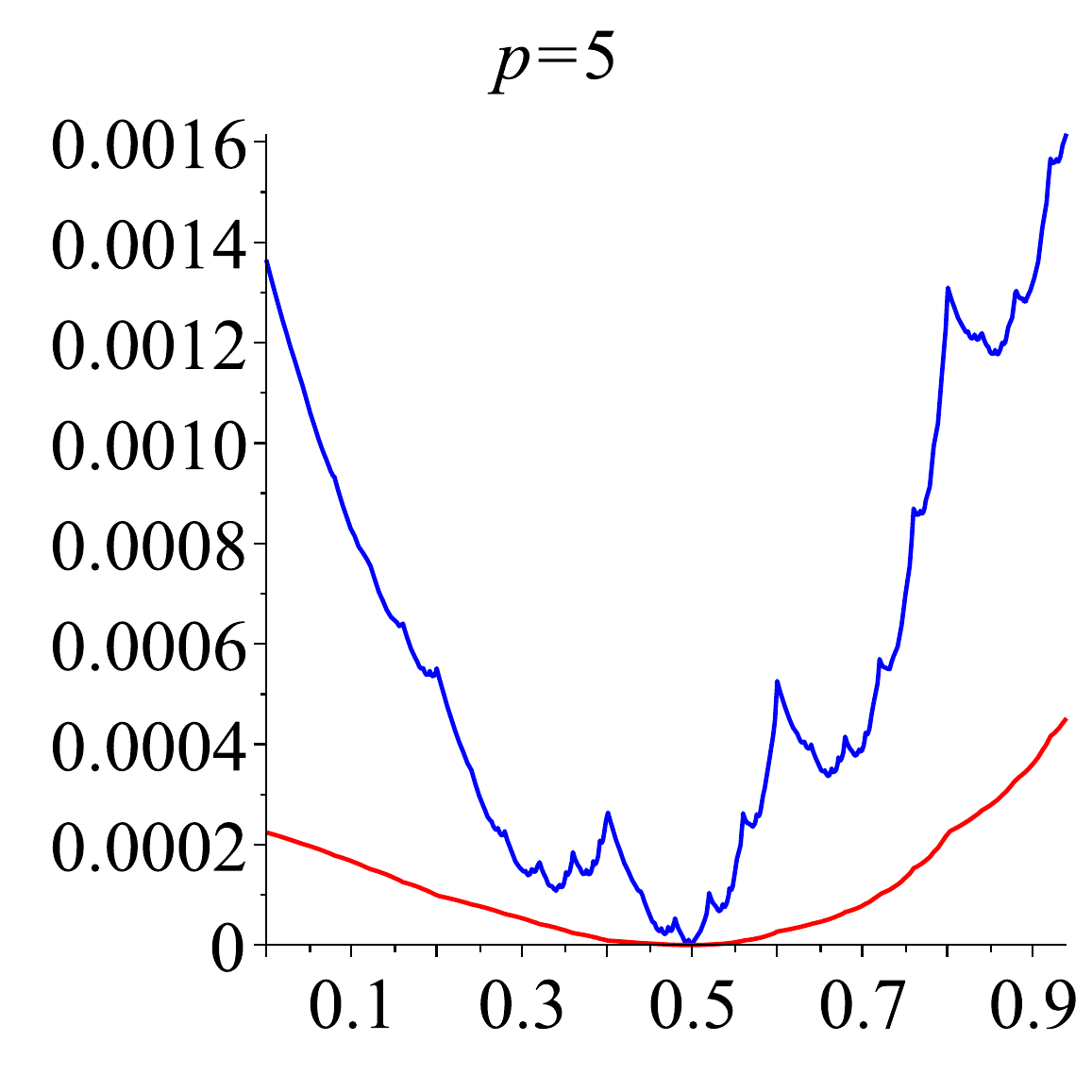}
		& \includegraphics[scale=0.2]{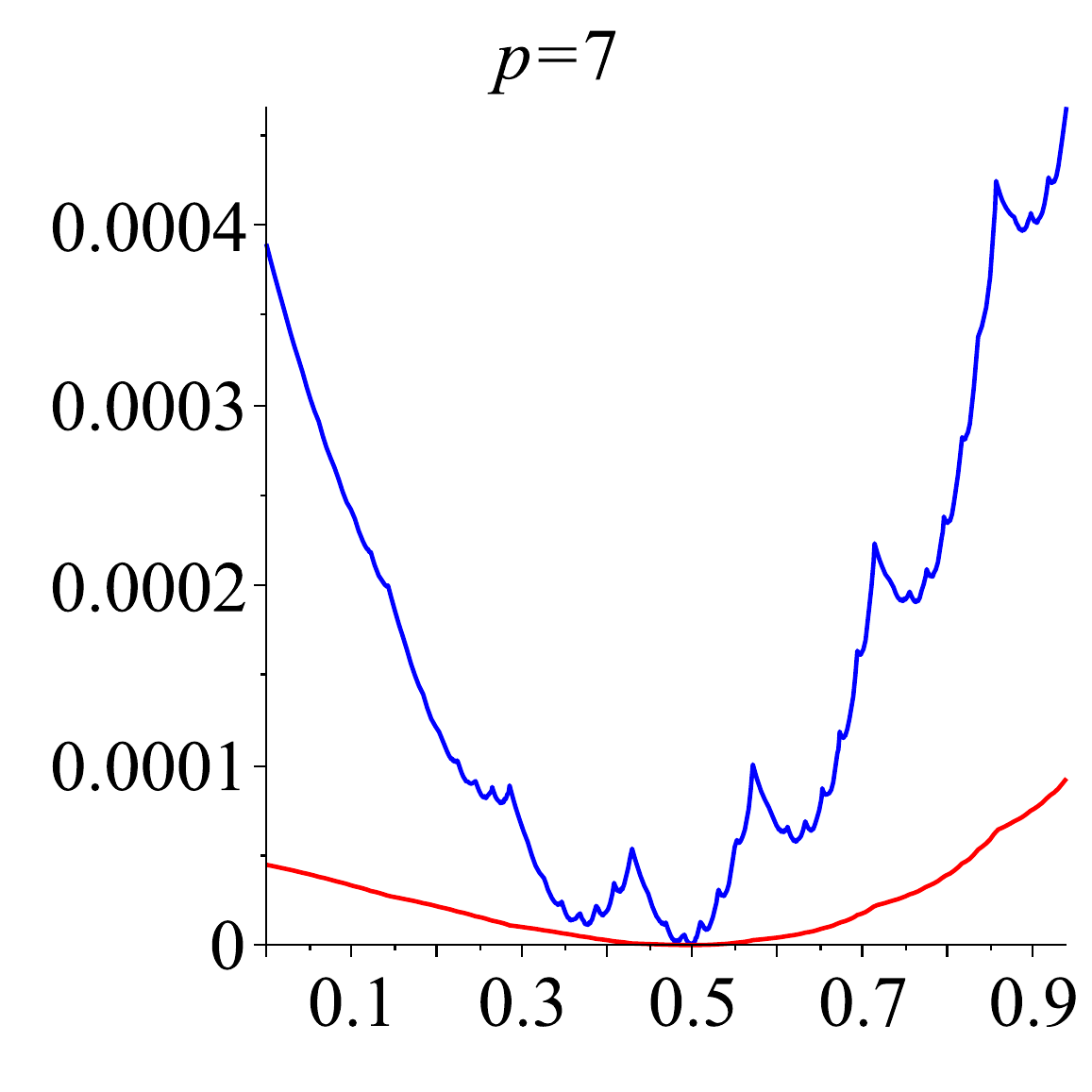}
	\end{tabular}
\end{center}
\caption{Graphical rendering of $\mathcal{Q}_{s_1}(t)$ (in blue) and 
$E_{s_1,1}(t)$ (in red) for $p=3,5,7$.}
\label{Q357}
\end{figure} 

\begin{figure}
\begin{center}	
	\begin{tabular}{ccc}
		\includegraphics[scale=0.2]{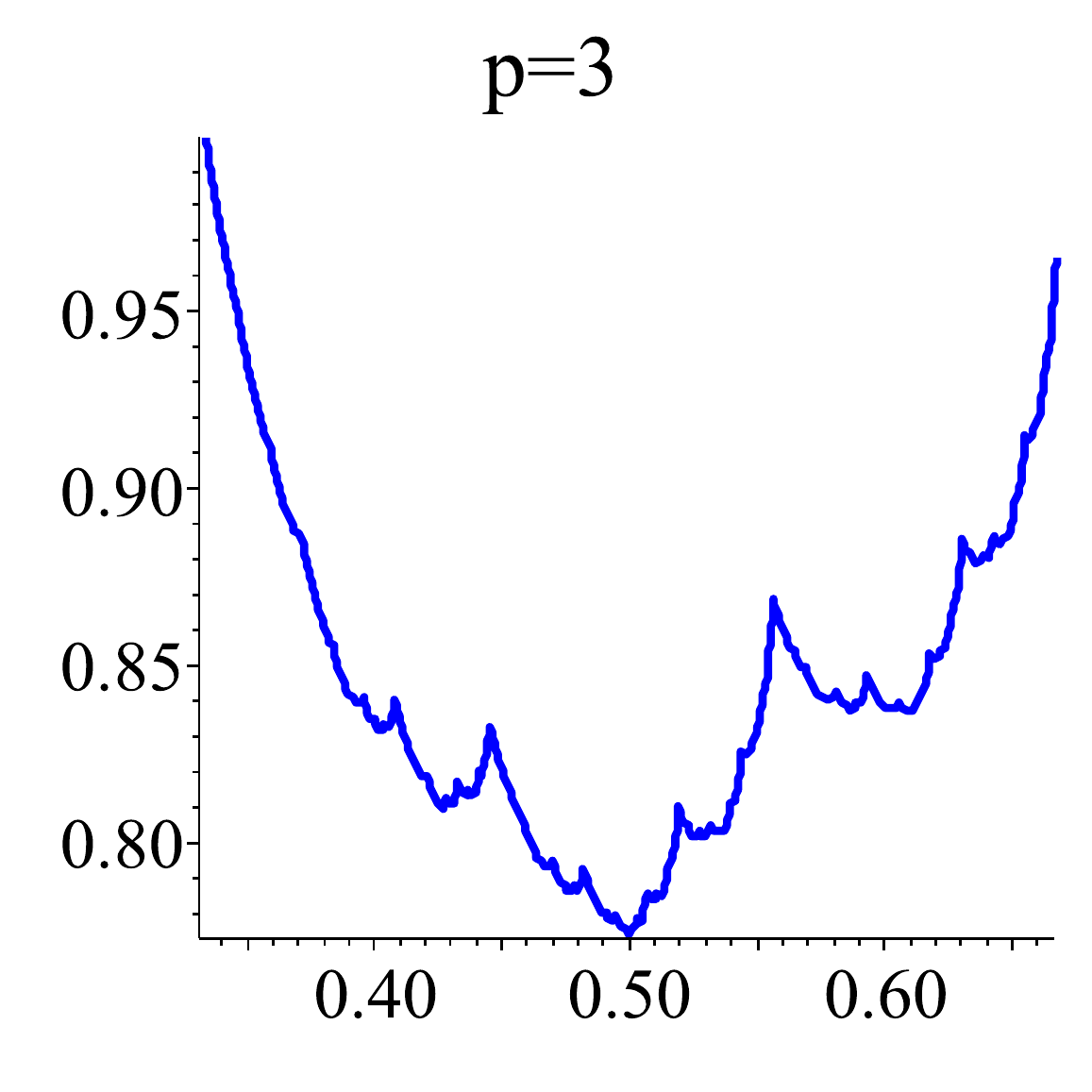}
		& \includegraphics[scale=0.2]{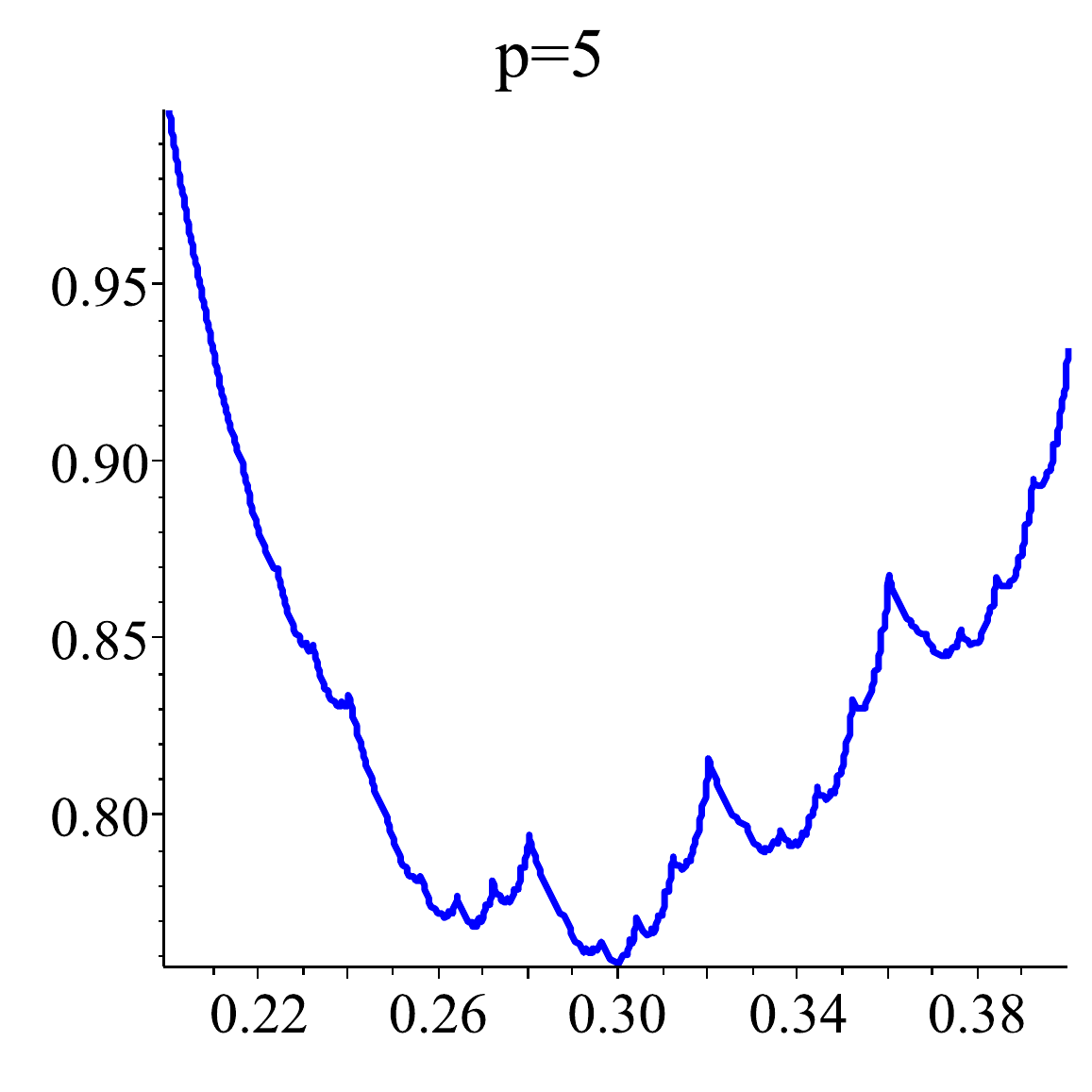}
		& \includegraphics[scale=0.2]{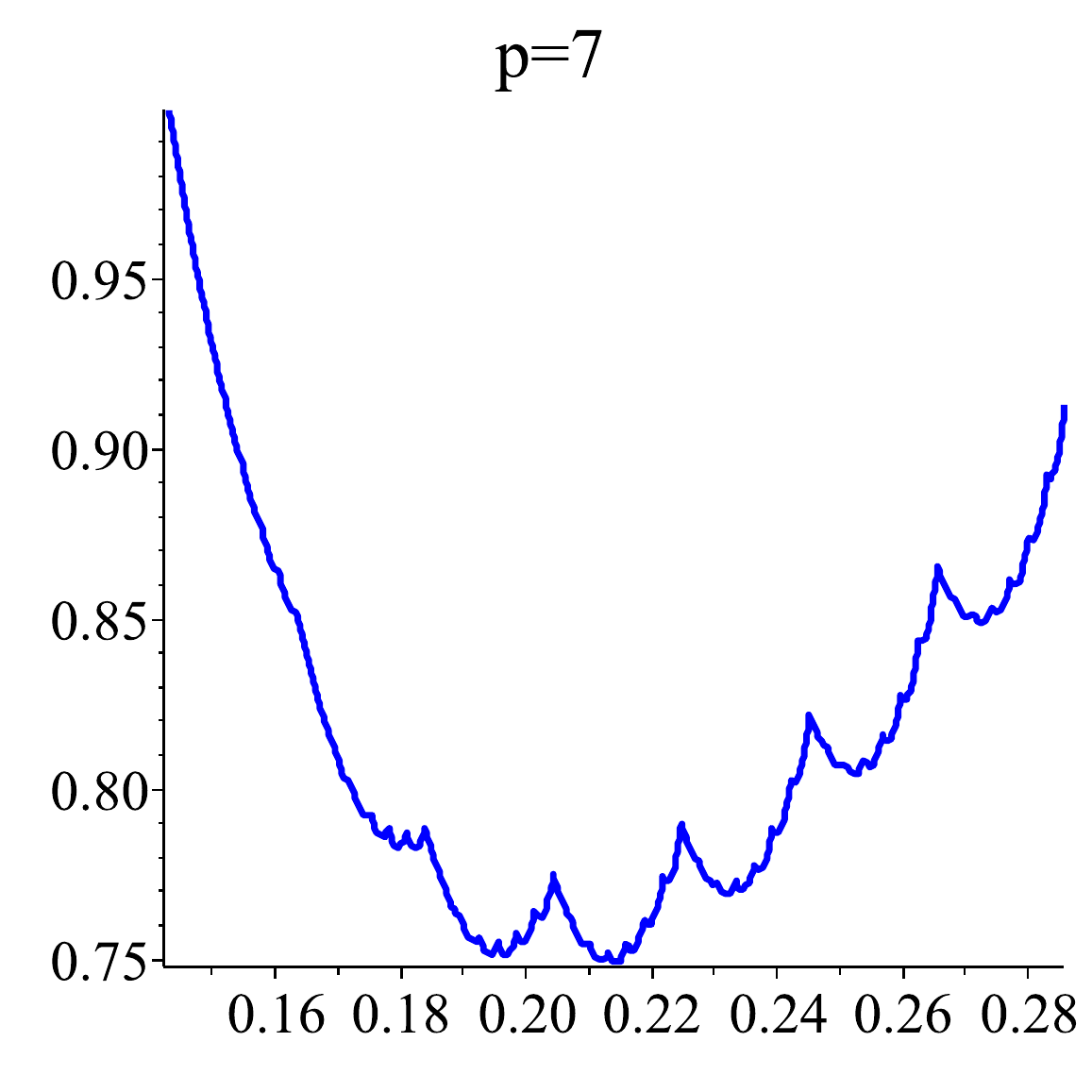}
	\end{tabular}
\end{center}
\caption{A closer look at the fluctuations of $G$ in the smaller
interval $\bigl[\frac{1}{p}, \frac{2}{p}\bigr]$ for
$p=3,5,7$.}\label{G357c}
\end{figure} 

\section{Proof of \refT{T1} for $p\ge11$}
\label{Spf2} 

When $p\ge11$, we take $M=2$, $k=0$ and 
$\mu = \hat{s}_p(\xi,\eta)=
\frac{2\xi+1}{2p}-\frac{\eta}{p^2}$ in Theorem~\ref{T4}, so that 
$a_1=\xi$, $a_2=\frac{p-1}{2}-\eta$
(recall \eqref{E:hat-sp2} and \eqref{t46}).
Then \eqref{tau} yields
\begin{align}
  	\tau_2=\frac{(\xi+1)(p+1-2\eta)}{2A^2},
\end{align}
and $\varphi(\mu)$ is given in \eqref{E:phi-hats}. Thus
\eqref{t42} and \eqref{t44} yield
\begin{align}
	\frac{p^2(\mathcal{Q}_{\mu}(t)-E_{\mu,0}(t))}
	{\rhop \varphi(\mu)}
	&= C(t)\left(\varphi(t)
	-\varphi\left(\frac{1}{2}\right)\right)	
	-\left(t-\frac{1}{2}\right),
\end{align}
where
\begin{align}
	C(t) := \frac{\tau_2^{}}{\rhop \varphi(\mu)}
	\lrpar{p^2\mu - \rhop
	\left(t-\frac{1}{2}\right)}.
\end{align}

\subsection{$p=11, 13, 17, 19, 23$: $(\xi,\eta)=(1,1)$}

For $p=11$, $s_1=\frac{3}{22}=0.13636\dots$ is no longer the minimum
point of $G$; see Figures \ref{p11Q} and \ref{Gp11}. In this case,
Wilson \cite{wi1} conjectured (in an equivalent form) that the
minimum occurs at $s_2:=s_1-\frac{1}{p^2}=0.12809\dots$, which yields
his conjecture for $\gb_{11}$ in \eqref{wc3}; see \refL{L:G-hat-s}.
(In base $11$, $s_2=0.14555\dots$.) Numerically, we have 
$G(s_1)=0.7386\ldots$ and $G(s_2)=0.7364\ldots$.

To verify his conjecture, we apply \refT{T4} with $M=2, k=0$ and 
$(\xi,\eta)=(1,1)$; thus $\mu=s_2 = \frac{3}{2p}-\frac1{p^2}$, which 
gives $a_1=1$, $a_2=\frac{p-3}2$, and then by \eqref{t42} 
\begin{equation}
	\mathcal{Q}_{s_2}(t)
	=\frac{(p-1)s_2}{A^2}\left(\varphi(t)
	-\varphi\left(\frac{1}{2}\right)\right)
	-\frac{\rhop}{p^2}\varphi(s_2)
	\left(t-\frac{1}{2}\right),
\end{equation}
and by \eqref{t44} 
\begin{equation}
	E_{s_2,0}(t)
	=\frac{(p-1)\rhop}{p^{k+2}A^2}
	\left(\varphi(t)-\varphi\left(\frac{1}{2}\right)\right)
	\left(t-\frac{1}{2}\right);
\end{equation}
see Figure \ref{p11Q} for an illustration of the different effects 
for $\mathcal{Q}_{\mu}(t)$ and $E_{\mu,k}(t)$ between $s_1$ and $s_2$.

The same numerical recipes used in the previous section for $p=3, 5,
7$ applies here with $(N_1,N_2)=(40,148)$, which shows that
(\ref{t43}) holds for $\mu=s_2$ and $k=0$, but not for $\mu=s_1$.
Thus, \refT{T4} guarantees that $G(s_2)$ is the minimum of $G$ in the
interval $I_2=\left[\frac{15}{11^2}, \frac{16}{11^2}\right]$. The
same bounding techniques with $(N_3,N_4)= (32,236)$ also shows (see
Figure \ref{Gp11}) that $G(s)>G(s_2)$ outside
$\left[\frac{15}{11^2},\frac{16}{11^2}\right]$. Thus, $G(s_2)$ is the
minimum, and $\gb_{11}=G(s_2)$ (see \eqref{wc3}).

\begin{figure}
\begin{center}	
	\begin{tabular}{cc}
		\includegraphics[scale=0.25]{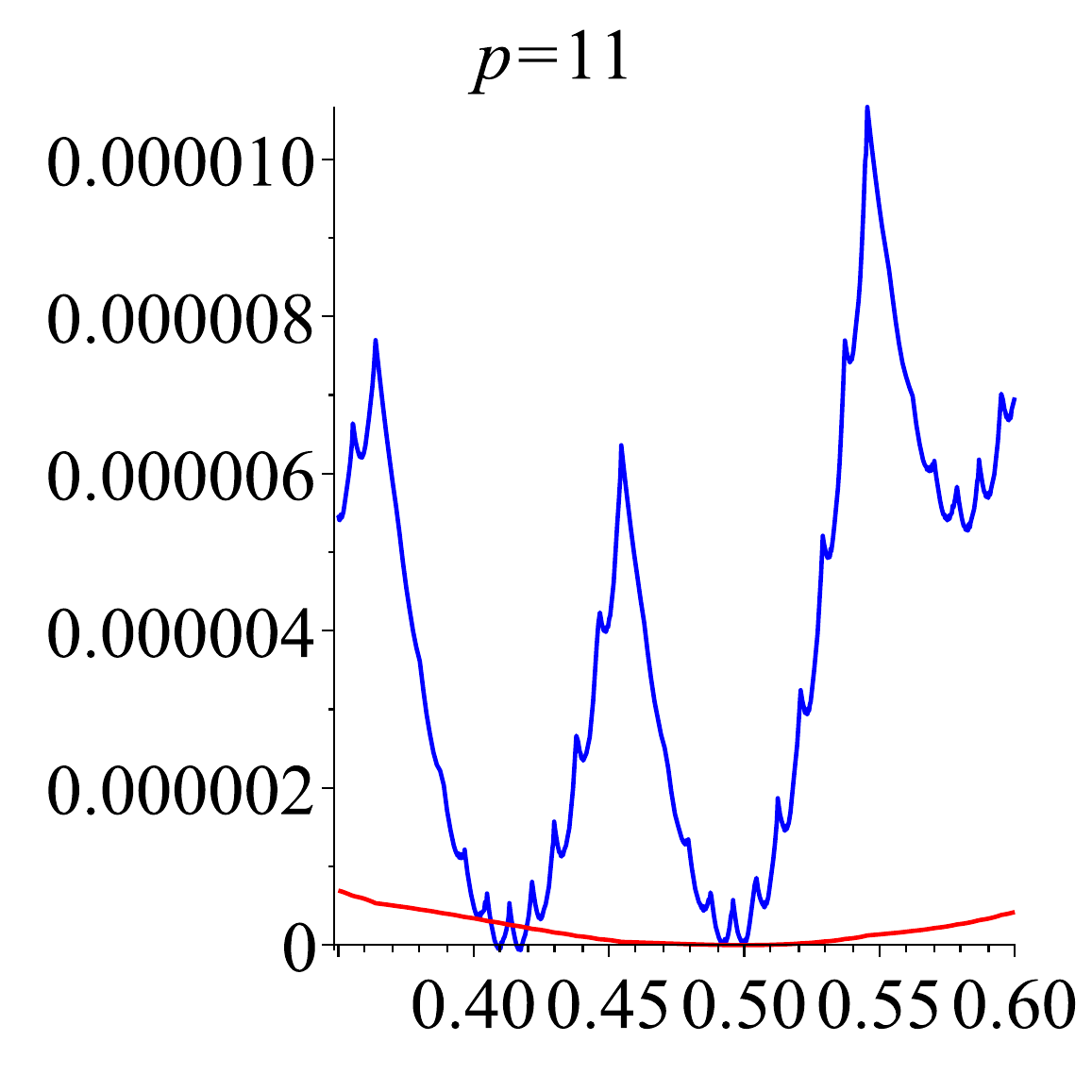}
		& \includegraphics[scale=0.25]{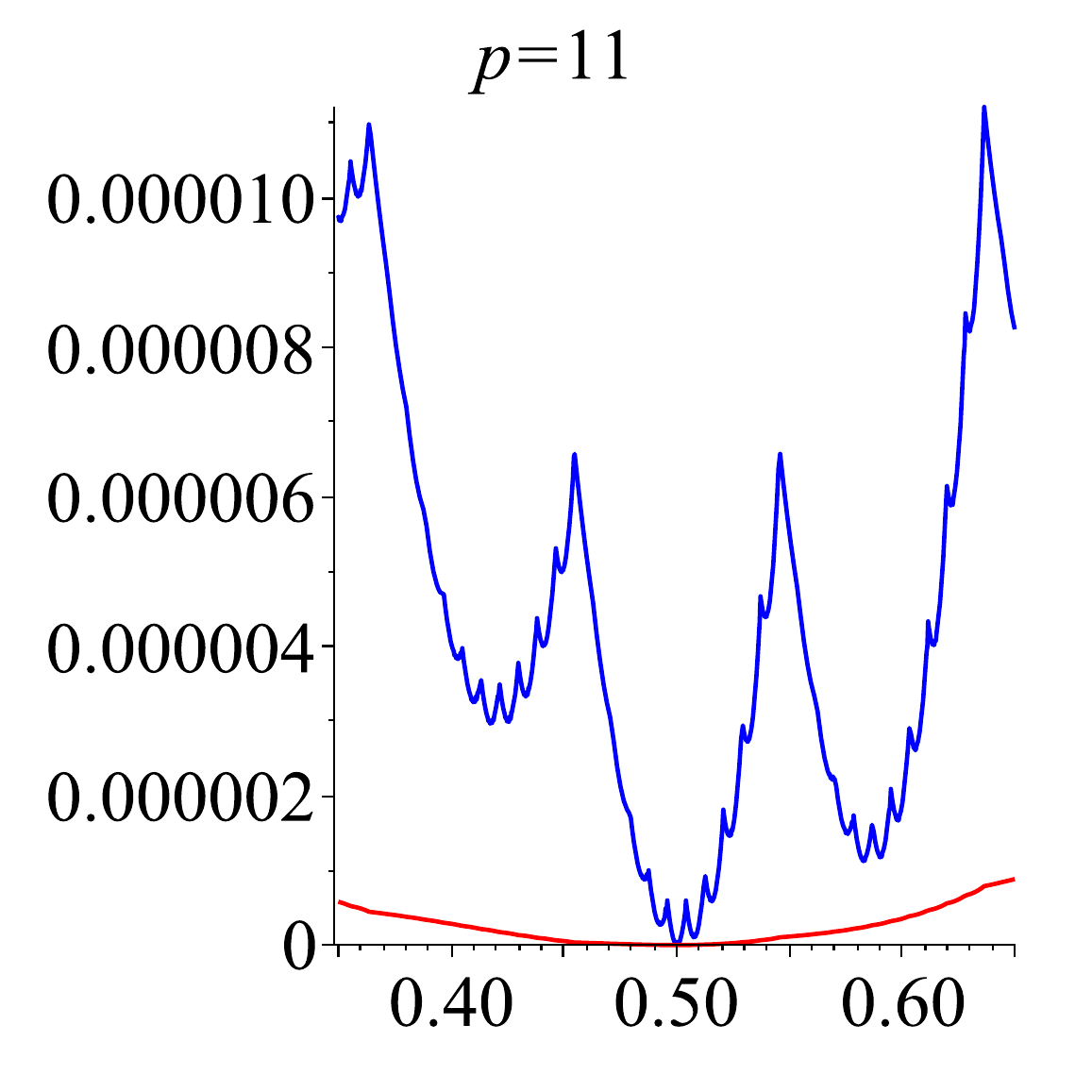}\\
		$\mathcal{Q}_{s_1}(t)$ vs $E_{s_1,0}$
		& $\mathcal{Q}_{s_2}(t)$ vs $E_{s_2,0}(t)$
	\end{tabular}
\end{center}
\caption{$p=11$: $\mathcal{Q}_{s_i}(t)$ (in blue) and $E_{s_i,0}$  (in red) for $i=1,2$.}\label{p11Q}
\end{figure} 

\begin{figure}
\begin{center}	
	\begin{tabular}{cc}
		\includegraphics[scale=0.25]{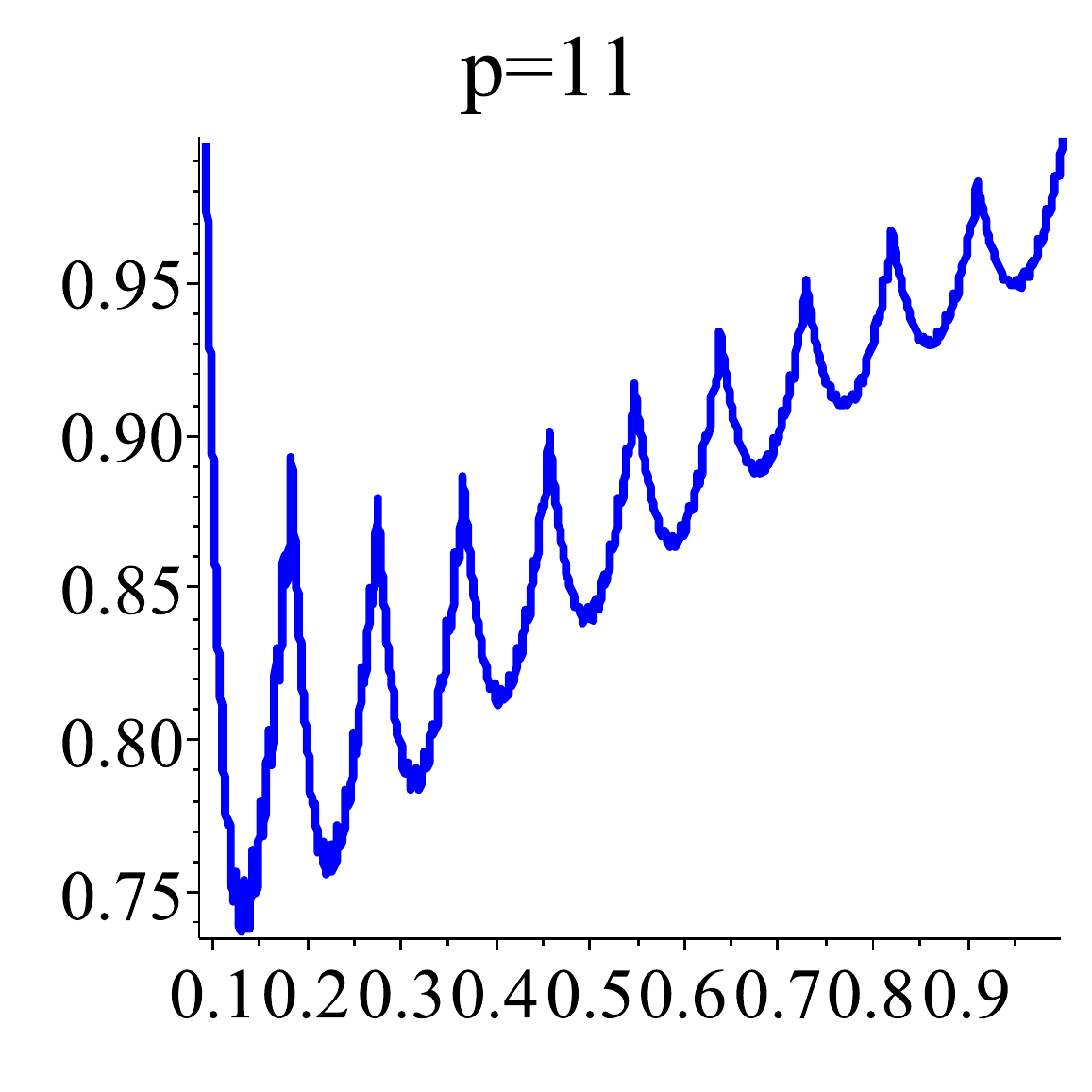}
		& \includegraphics[scale=0.25]{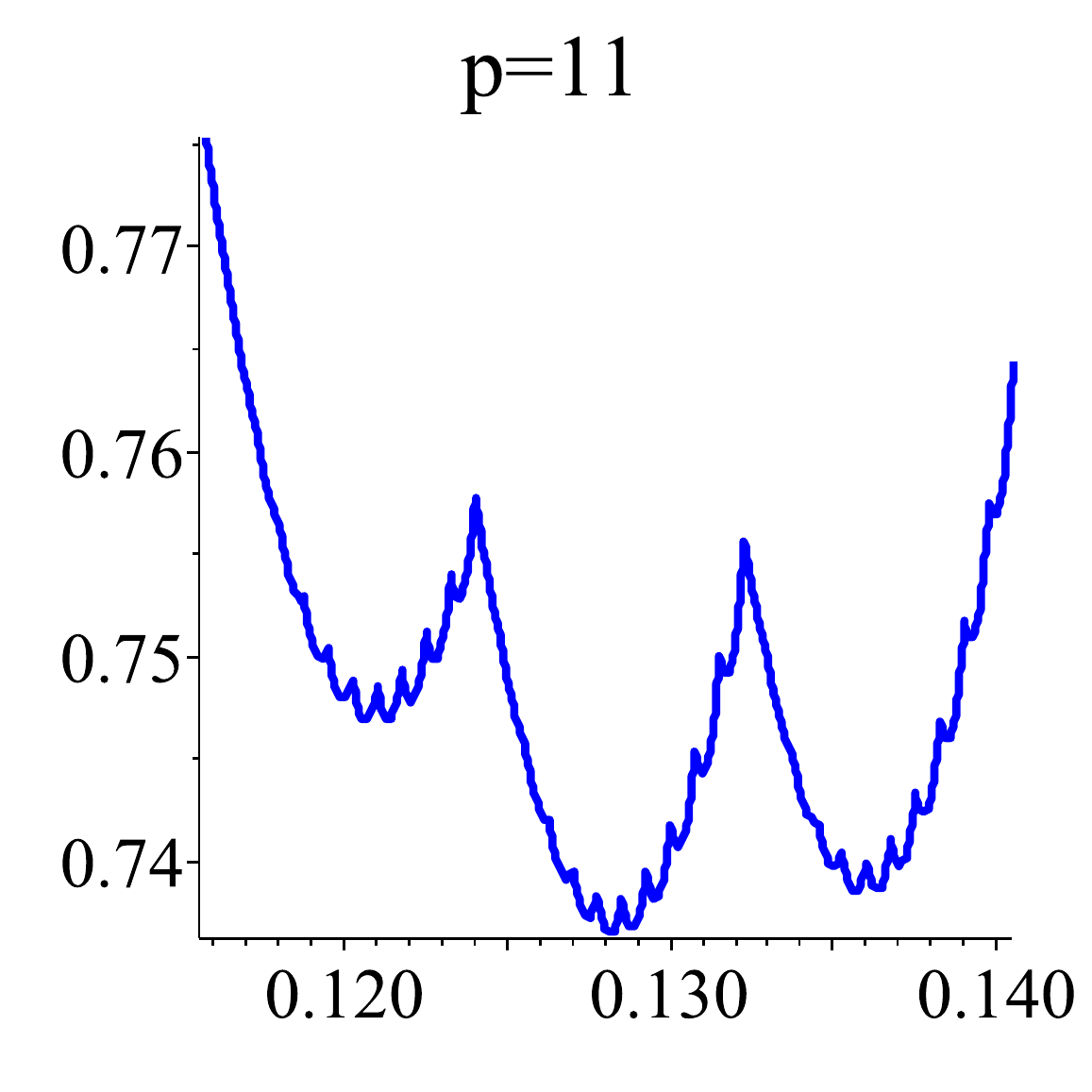}
	\end{tabular}
\end{center}
\caption{$p=11$: $G(x)$ for $x\in\left[ \frac{1}{11},1\right]$ (left)
and $\left[ \frac{14}{11^2}, \frac{17}{11^2}\right]$ (right),
respectively.}
\label{Gp11}
\end{figure} 

Similarly, the same procedure applies to $p=13,17,19,23$ with
\begin{center}
\begin{tabular}{c|c}
$p$ & $(N_1,N_2,N_3,N_4)$ \\ \hline
$13$ & $(62, 131, 53, 373)$ \\
$17$ & $(135, 132, 134, 331)$ \\
$19$ & $(211, 144, 257, 1517)$ \\
$23$ & $(611, 151, 992, 6812)$
\end{tabular}	
\end{center}
for which $G$ also attains its minimum $\gb_p$ at $s_2$; see Figures 
\ref{G13} and \ref{Q13}. 

This proves \refT{T1} for $11\le p\le 23$.
Using \refL{L:G-hat-s}, we obtain the 
values of $\beta_p$:
\begin{center}
\renewcommand{\arraystretch}{2} 
\begin{tabular}{ccccc}
$p$ & $13$ & $17$ & $19$ & $23$ \\ \hline
$\beta_p$ 
& $\dfrac{124}{91}\Bigl(\dfrac{26}{37}\Bigr)^{\varrho_{13}}$
& $\dfrac{71}{51}\Bigl(\dfrac{34}{49}\Bigr)^{\varrho_{17}}$ 
& $\dfrac{533}{380}\Bigl(\dfrac{38}{55}\Bigr)^{\varrho_{19}}$
& $\dfrac{261}{184}\Bigl(\dfrac{46}{67}\Bigr)^{\varrho_{23}}$
\end{tabular}	
\end{center}

\subsection{$p=29$: $(\xi,\eta)=(2,1)$}

For $p=29$, $s_2=\hat s_p(1,1)=\frac{3}{2p}-\frac{1}{p^2}$ is no longer the
global 
minimum point of $G(s)$. Instead $\hat s_p(1,2)=\frac{3}{2p}-\frac{2}{p^2}$
gives a 
smaller value of $G(s)$, and an even smaller value of $G$ is reached
at $s_3:=\hat{s}_p(2,1)=\frac{5}{2p}-\frac{1}{p^2}$, which can be proved to
be the 
minimum point by \refT{T4} with the same numerical recipes used 
above. 

When it comes to numerical check, a direct use of the preceding
numerical recipes gives the (minimum) numbers of partitions required
in each of the intervals $[0,\frac12-\frac1{2p}]$,
$[\frac12+\frac1{2p},1]$, $[\frac1p,\frac5{2p}-\frac3{2p^2}]$ and
$[\frac5{2p}-\frac1{2p^2}]$: $(N_1,N_2,N_3,N_4)= (3011, 216, 14996,
11942)$, respectively, which are somewhat too large. 
We used above subintervals of the same length, but this is not optimal, 
and the
computational complexity can be reduced by the following procedure:
instead of fixing first the interval and then finding a large enough 
number of subintervals (of equal length)
such that the monotonicity inequality holds in each of the subintervals,
we fix first $N$, the number
of subintervals to be processed in each step, 
and then check either from the left end
or the right end of the interval how far towards the other end of the
interval we can go with $N$ subintervals of the same size such that the
monotonicity inequality holds in all subintervals.  
Then repeat the same procedure until reaching the other end of the
interval. Alternatively, a binary splitting technique of the target
interval can be used to identify a range where $N$ partitions
suffice.

For example, to check if $\mathcal{Q}_{s_3}(t)>E_{s_3,0}(t)$ holds in
the interval $[0,\frac12-\frac1{2p}]$, we choose, say $N=p^2=841$,
which then suffices if we partition first $[0,\frac12-\frac1{2p}]$
into the two subintervals $[0,\frac12-\frac4p]$ and
$[\frac12-\frac4p, \frac12-\frac1{2p}]$, and then further partition
each into $N+1$ smaller subintervals before checking the monotonicity
inequality. Similarly, instead of using $N_3=14996$ for the interval
$[\frac1p, \frac5{2p}-\frac{3}{2p^2}]$, we choose again $N=p^2$ and
split first this interval into $[\frac1p, \frac3{p^2}]$ and
$[\frac5{2p}-\frac{3}{p^2},\frac5{2p}-\frac{3}{2p^2}]$ before the
numerical check in both subintervals. Finally, the use of the number
$N_4=11942$ for the interval $[\frac5{2p}-\frac1{2p^2},1]$ can be
replaced by taking $N=p^2$ and splitting
$[\frac5{2p}-\frac1{2p^2},1]$ into $[\frac5{2p}-\frac1{2p^2},\frac
4p]$ and $[\frac4{p},1]$.

\subsection{Primes from $31$ to $113$: $\xi=2$}

Exactly the same method of proof used above applies to higher values
of $p$ with the minimum point of $G$ in $[\frac{1}{p},1]$  given in
\eqref{T1-3-0}. The two-stage partitioning procedure is 
computationally more efficient. For example, when $p=113$, we have:
\begin{center}
\renewcommand{\arraystretch}{1.5}		
\begin{tabular}{cccc}
variable & 
Interval & \makecell{first \\partition} & \makecell{split each \\
into $N$ subintervals\\
(equal spacing)}\\ \hline
\multirow{2}{*}{$t$} &
$[0,\frac12-\frac1{2p}]$ & 
$[0,\frac12-\frac3p]\cup [\frac12-\frac3p,\frac12-\frac1{2p}]$
& $N=500$ \\ &
$[\frac12+\frac1{2p},1]$ &
$[\frac12+\frac1{2p},\frac12+\frac7{2p}] \cup 
[\frac12+\frac1{2p},1]$ & $N=700$\\ \hline
\multirow{2}{*}{$s$} &
$[\frac1p,\frac{5}{2p}-\frac{11}{2p^2}]$ & 
$[\frac1p,\frac5{2p}-\frac{11}{p^2}]\cup 
[\frac5{2p}-\frac{11}{p^2},
\frac5{2p}-\frac{11}{2p^2}]$ & $N=500$ \\  &
$[\frac5{2p}-\frac{9}{2p^2}, 1]$ & 
$[\frac5{2p}-\frac{9}{2p^2}, \frac5{2p}]$ $\cup$
$[\frac5{2p},\frac5{p}]$ $\cup$
$[\frac5{p},1]$ & $N=5000$ \\ \hline
\end{tabular}	
\end{center}

\begin{figure}
\begin{center}	
	\begin{tabular}{ccc}
		\includegraphics[scale=0.2]{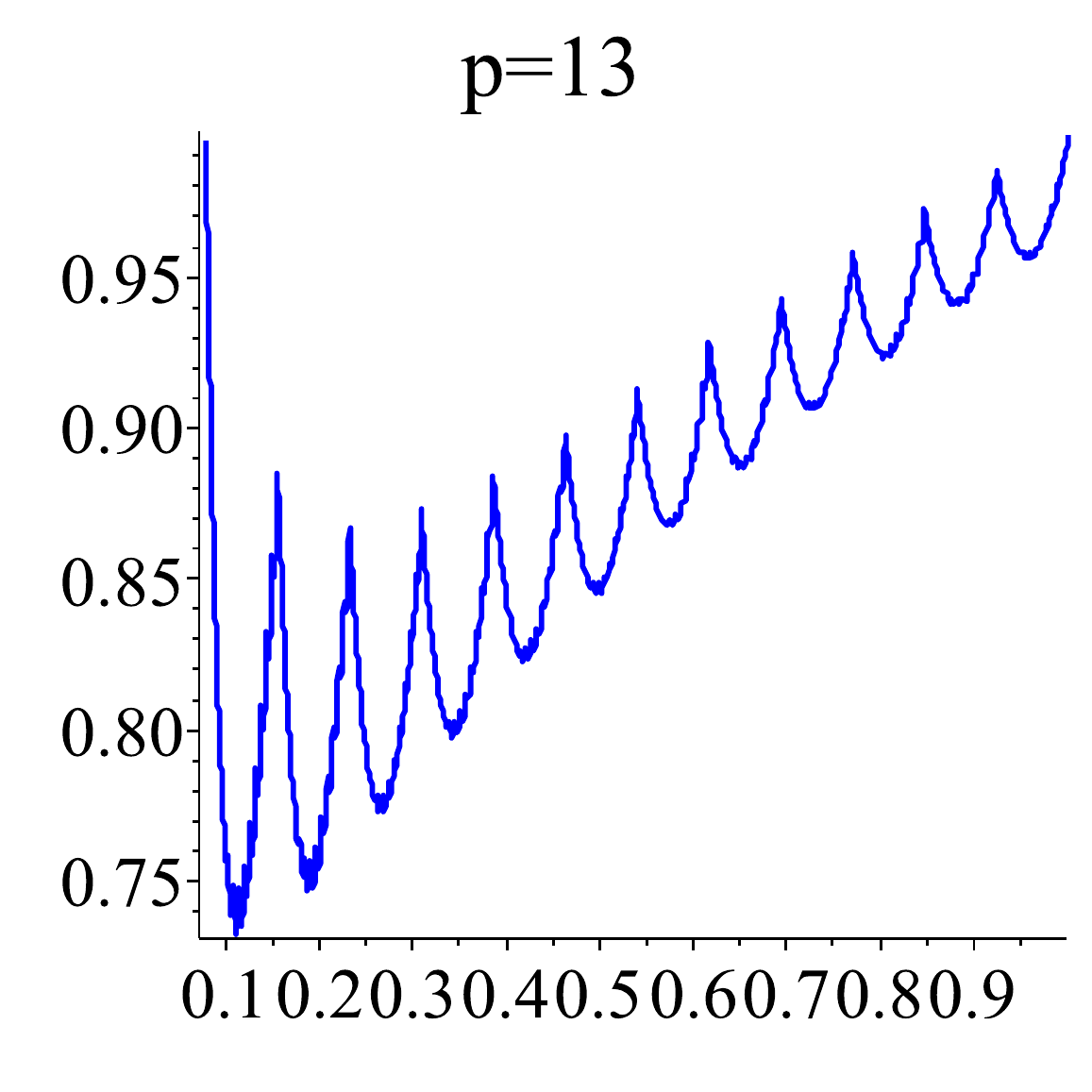}
		& \includegraphics[scale=0.2]{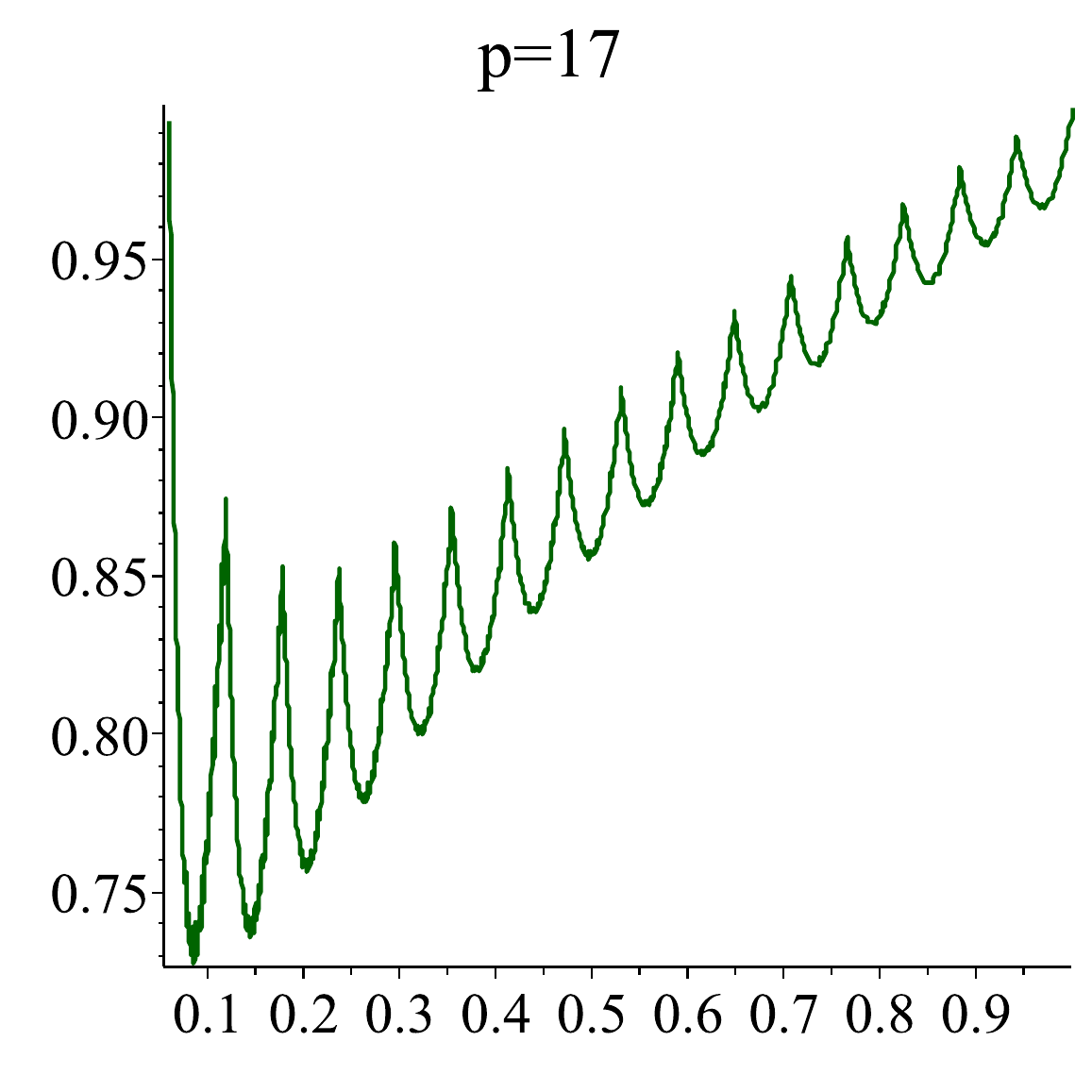}
		& \includegraphics[scale=0.2]{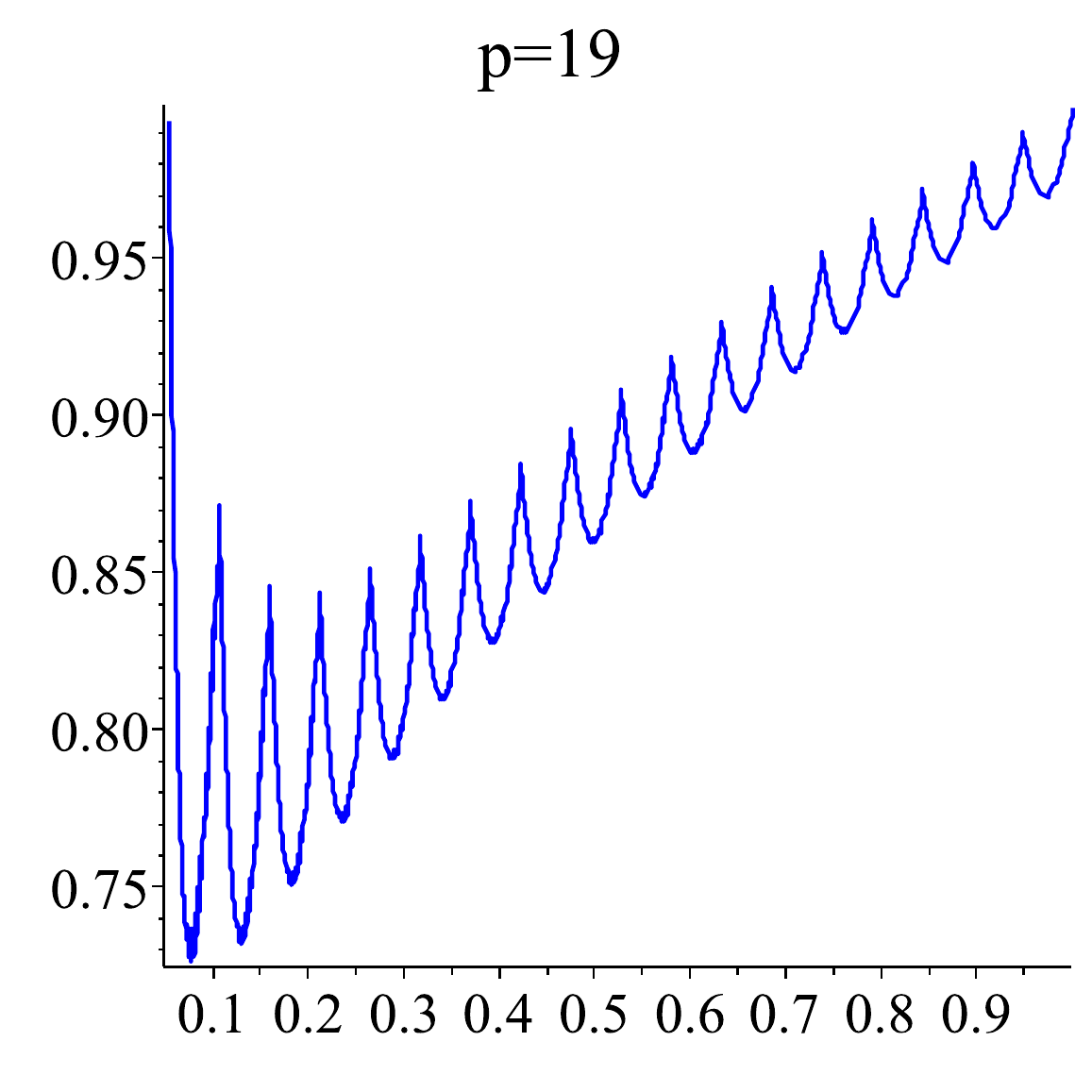}
	\end{tabular}
\end{center}
\caption{A graphical rendering of $G(s)$ for $p=13,17,19$.}
\label{G13}
\end{figure} 

\begin{figure}
\begin{center}	
	\begin{tabular}{ccc}
		\includegraphics[scale=0.2]{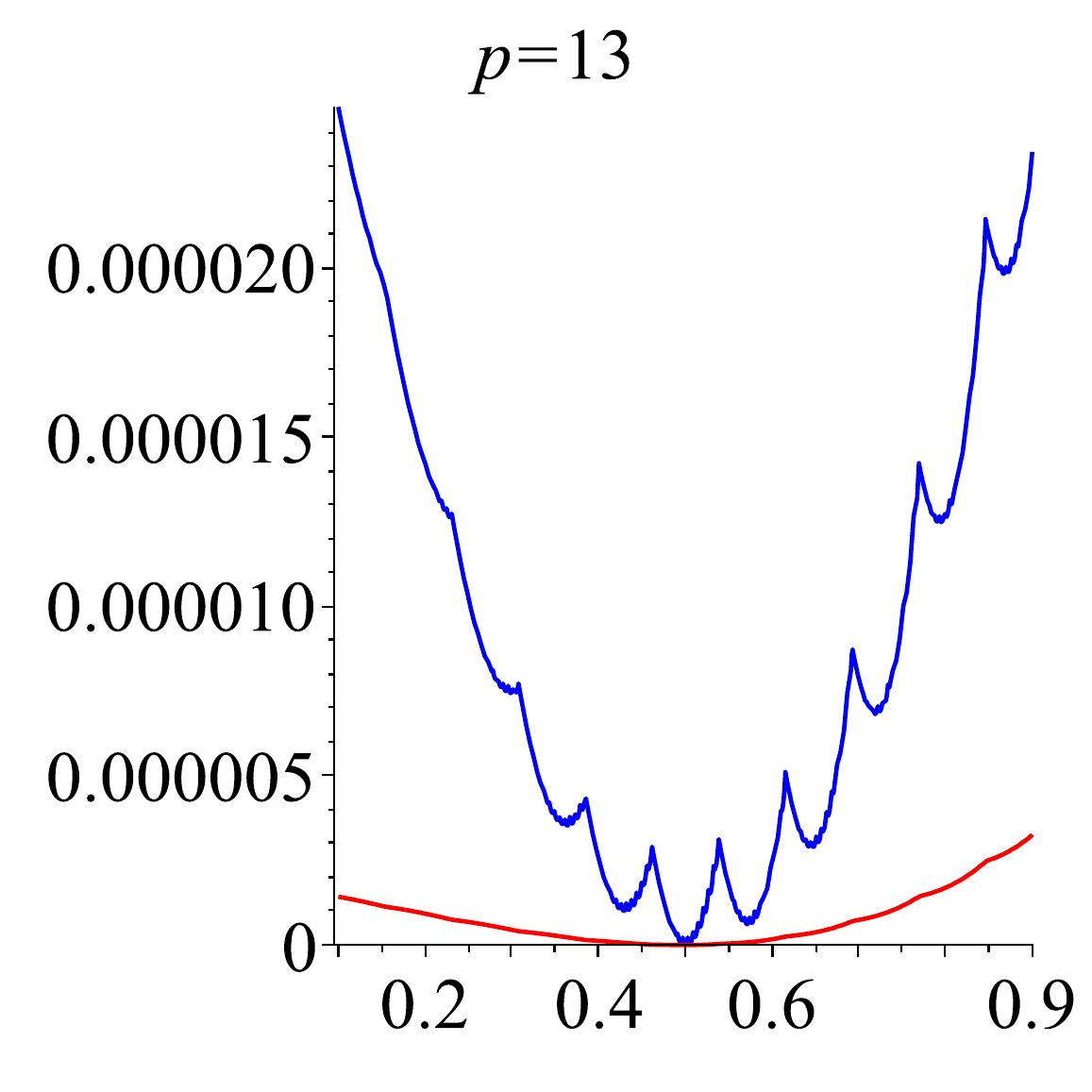}
		& \includegraphics[scale=0.2]{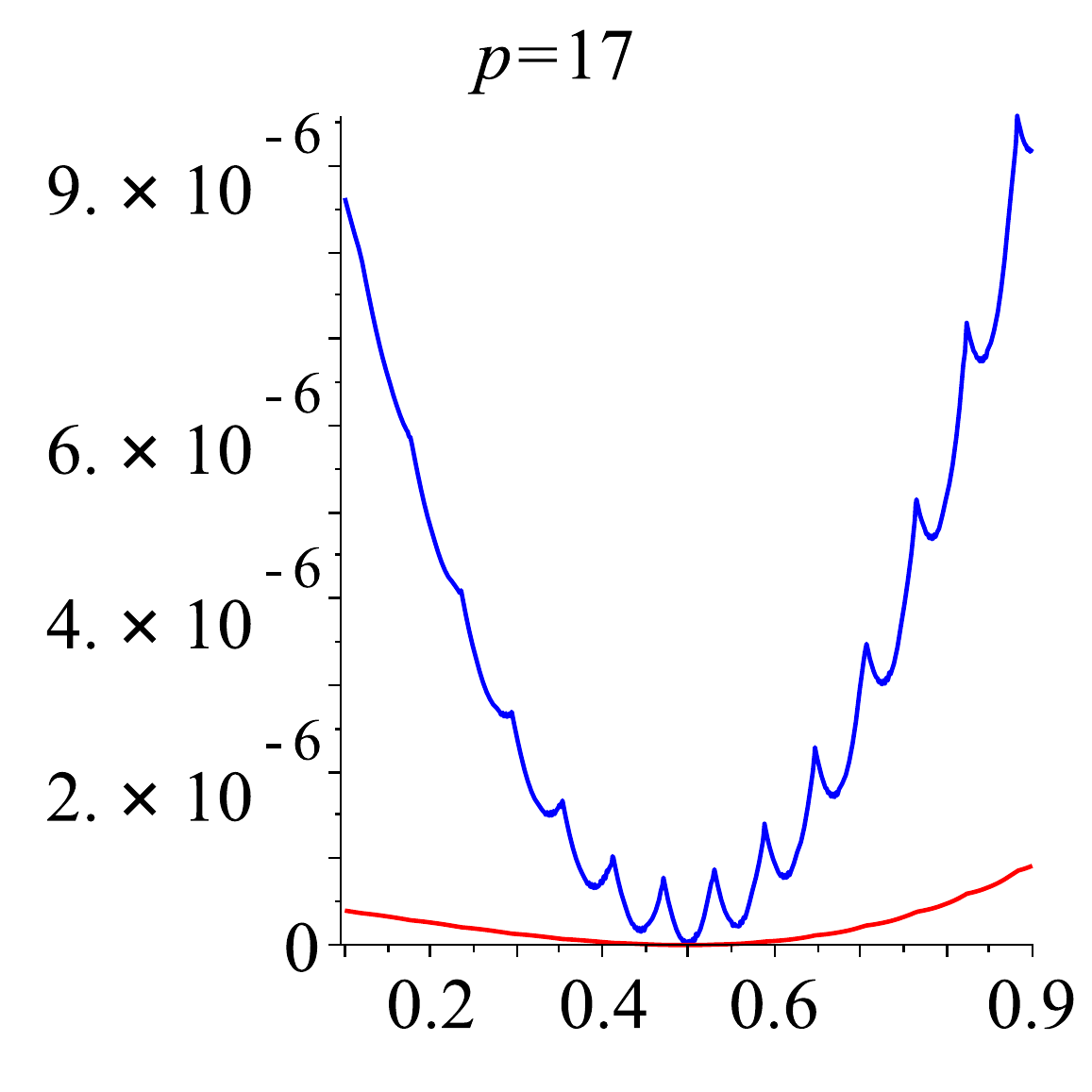}
		& \includegraphics[scale=0.2]{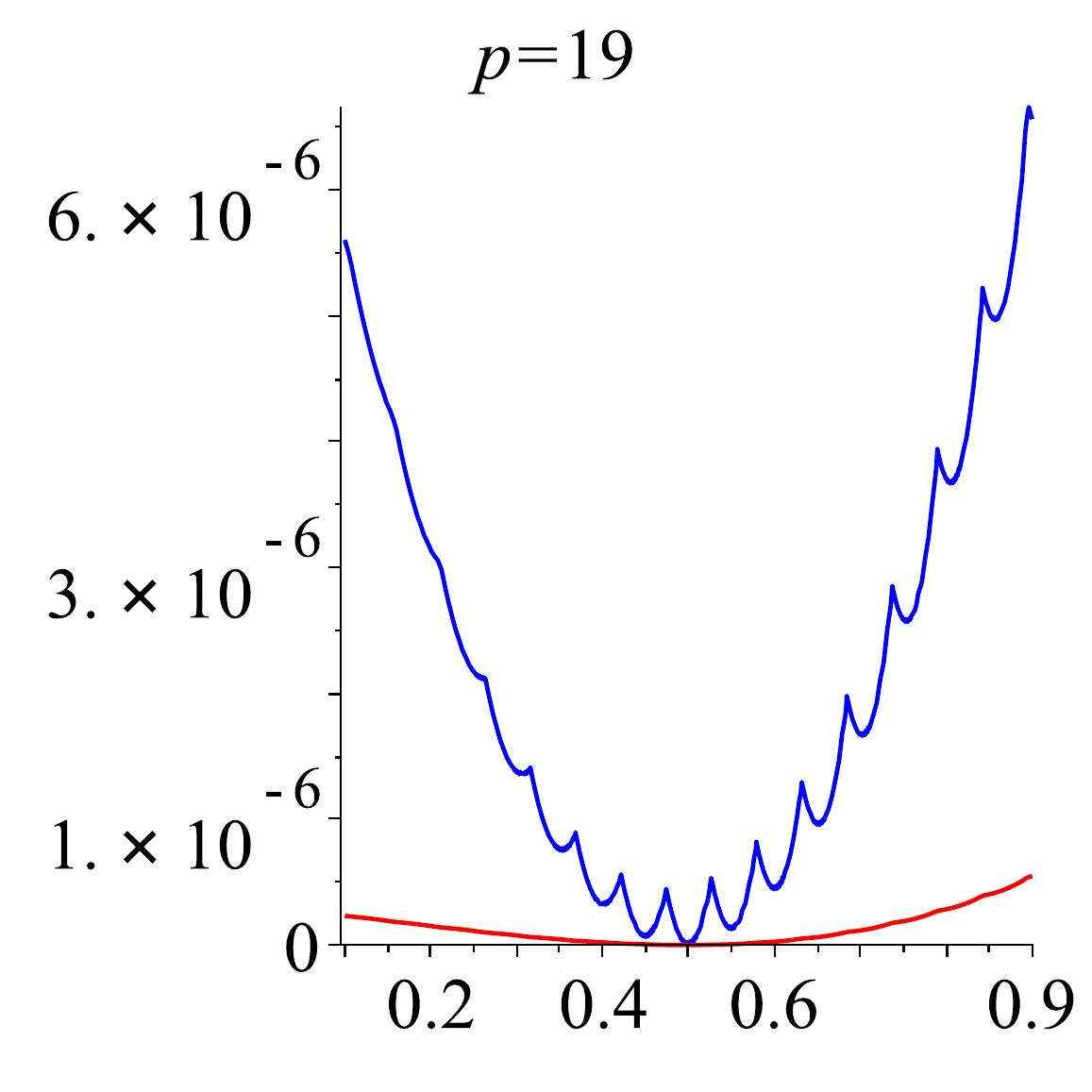} \\
		\includegraphics[scale=0.2]{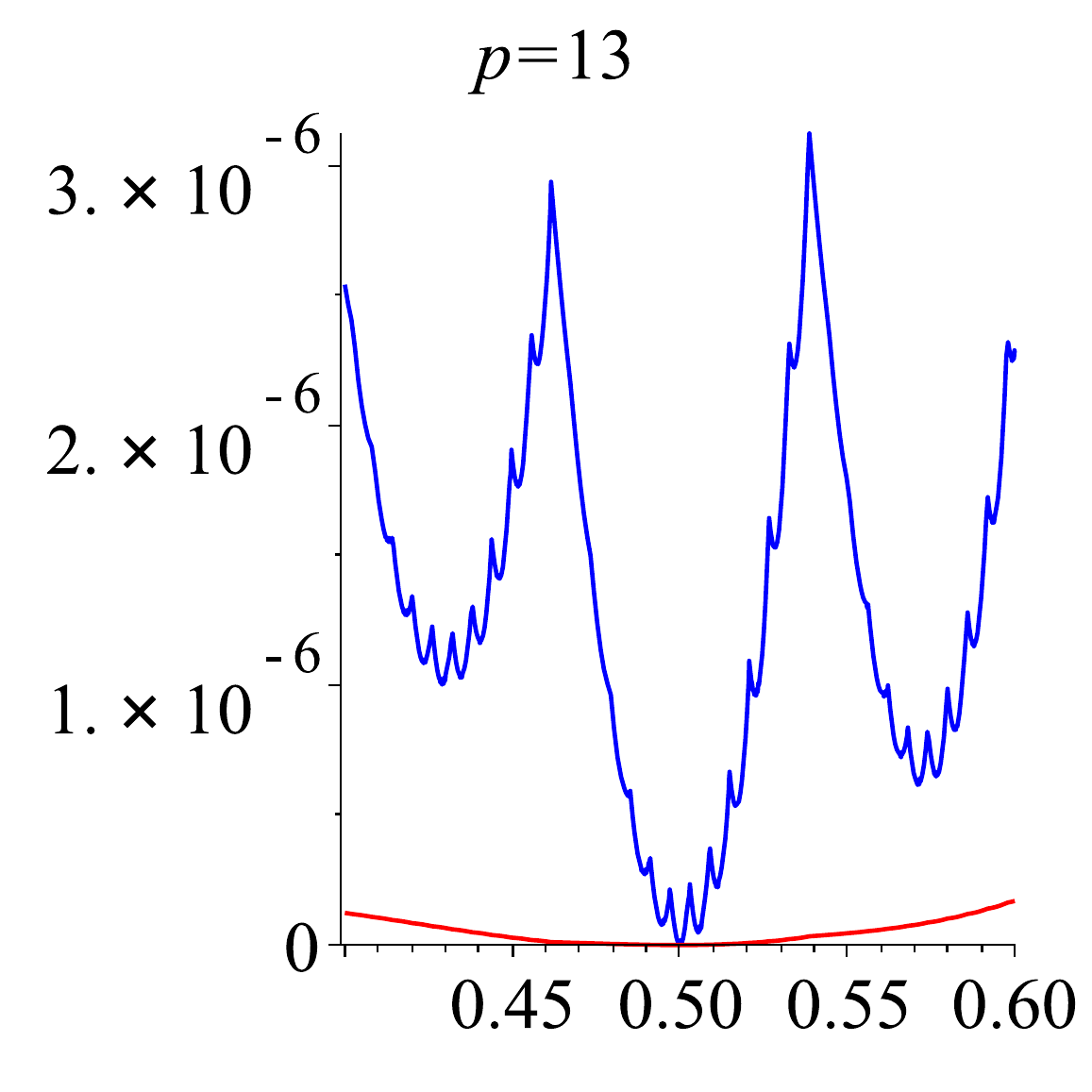}
		& \includegraphics[scale=0.2]{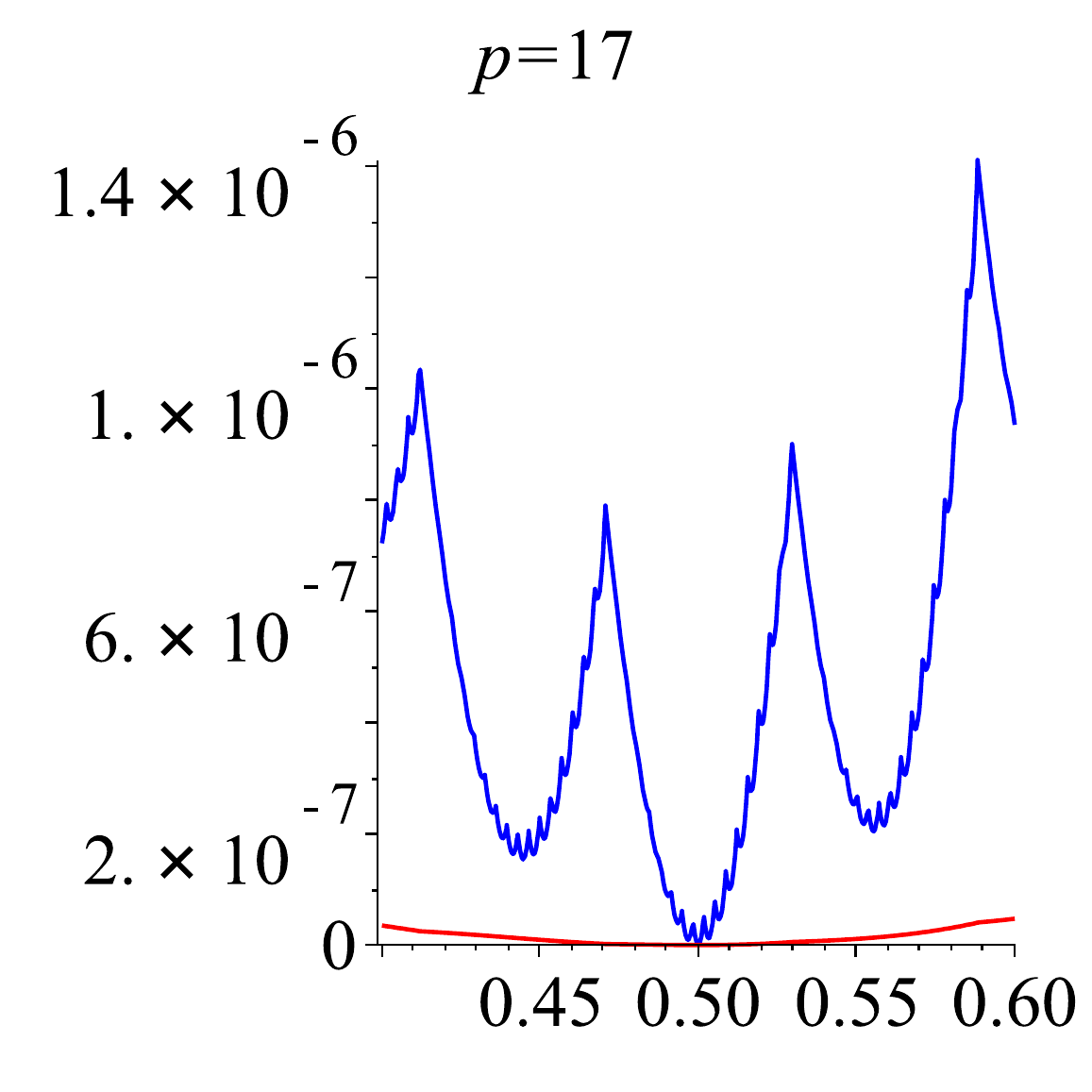}
		& \includegraphics[scale=0.2]{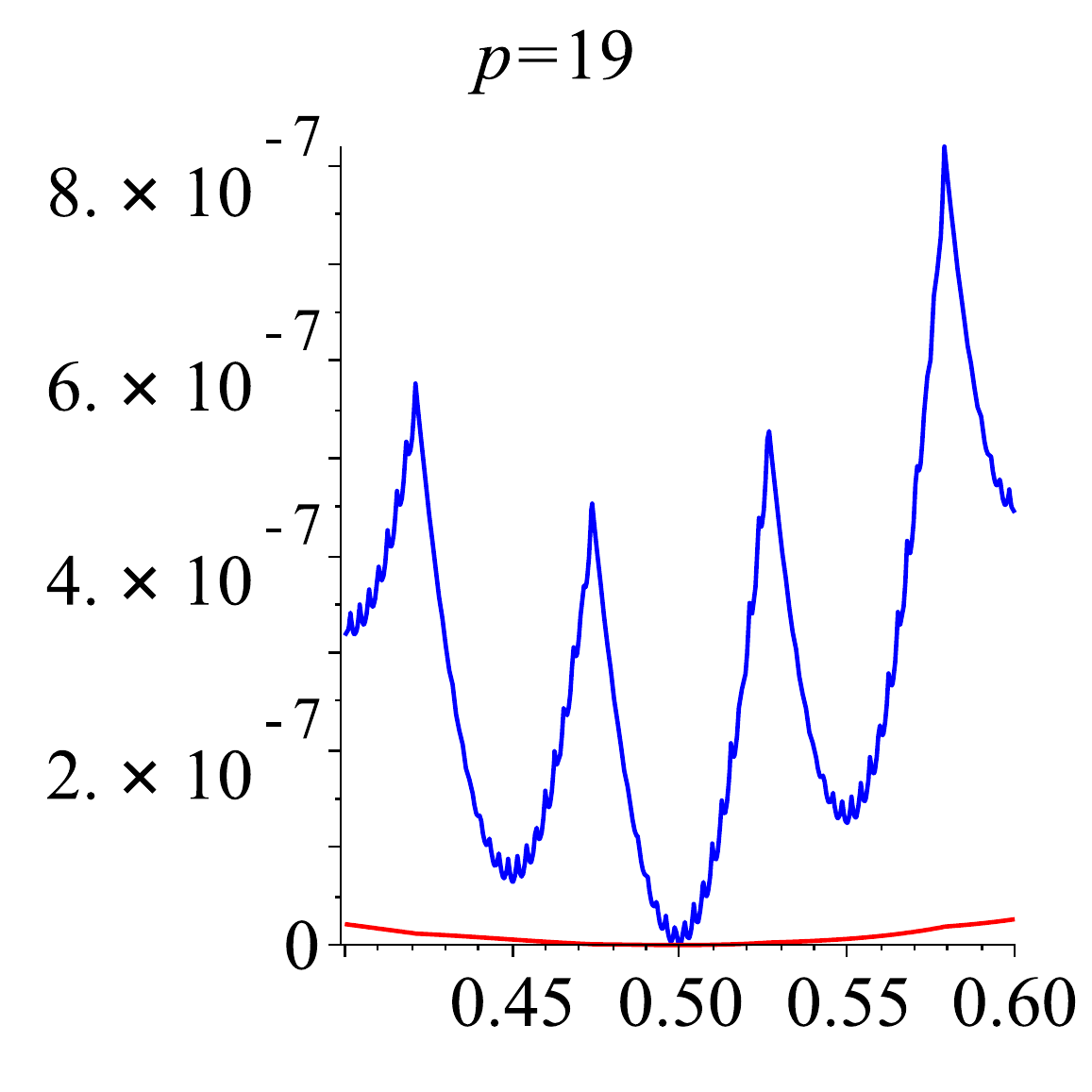}
	\end{tabular}
\end{center}
\caption{$\mathcal{Q}_{s_2}(t)$ (in blue) vs $E_{s_2,0}(t)$ (in red) 
for $p=13,17,19$.} \label{Q13}
\end{figure} 

\section{$p\ge127$}
\label{S:large-p}

In this section, we discuss briefly the extension of our approach to 
larger primes. 

\subsection{$p=127,\dots,2221$}\label{SS:127+}

The same approach used so far is readily extended to primes of larger
values. As far as our numerical check was conducted, the minimum 
$G(s)$ 
is always reached at
$s=\hat{s}_p=\frac{2\xi+1}{2p}-\frac{\eta}{p^2}$ for a suitable choice of
$(\xi,\eta)$; consequently, then $\gb_p$ is given by \eqref{E:Bml}.
In general, the first-stage partition can be chosen by
standard binary search. The following table shows the choices of
$(\xi,\eta)$ for $p\le 2221$.

\begin{table}[!ht]
\begin{center}
\renewcommand{\arraystretch}{1.5}	
\begin{tabular}{cccccc}
$3$--$7$ & $11$--$23$ & $29$ 
& $31$--$53$ & $59$--$79$ & $83$--$107$ \\ \hline 
$(1,0)$ & \textcolor{red}{$(1,1)$} & \textcolor{red}{$(2,1)$} 
& $(2,2)$ & $(2,3)$ & $(2,4)$ \\
$109$--$113$ & $127$--$139$ & $149$--$173$ 
& $179$--$199$ & $211$--$241$ & $251$--$277$ \\ \hline
\textcolor{red}{$(2,5)$} & \textcolor{red}{$(3,5)$} & $(3,6)$ 
& $(3,7)$ & $(3,8)$ & $(3,9)$ \\
$281$--$311$ & $313$--$347$ & $349$--$383$ 
& $389$--$419$ & $421$--$449$ & $457$--$487$ \\ \hline
$(3,10)$ & $(3,11)$ & $(3,12)$
& $(3,13)$ & $(3,14)$ & \textcolor{red}{$(3,15)$}\\
$491$--$509$ & $521$--$547$ & $557$--$587$
& $593$--$619$ & $631$--$661$ & $673$--$701$\\ \hline
\textcolor{red}{$(4,15)$} & $(4,16)$ & $(4,17)$ 
& $(4,18)$ & $(4,19)$ & $(4,20)$ \\	
$709$--$743$ & $751$--$787$ & $797$--$829$
& $839$--$863$ & $877$--$911$ 
& $919$--$953$ \\ \hline
$(4,21)$ & $(4,22)$ & $(4,23)$ 
& $(4,24)$ & $(4,25)$ & $(4,26)$ \\
$967$--$991$ & $997$--$1033$ & $1039$--$1069$
& $1087$--$1117$ & $1123$--$1163$ 
& $1171$--$1201$ \\ \hline
$(4,27)$ & $(4,28)$ & $(4,29)$
& $(4,30)$ & $(4,31)$ & $(4,32)$ \\
$1213$--$1249$ & $1259$--$1291$
& $1297$--$1327$ & $1361$--$1373$ 
& $1381$--$1423$ & $1427$--$1459$ \\ \hline	
$(4,33)$ & $(4,34)$ & $(4,35)$
& $(4,36)$ & $(4,37)$ & $(4,38)$\\
$1471$--$1511$ & $1523$--$1553$ 
& $1559$--$1597$ & $1601$--$1637$
& $1657$--$1669$ & $1693$--$1723$ \\ \hline
$(4,39)$ & $(4,40)$ & $(4,41)$ 
& $(4,42)$ & $(4,43)$ & $(4,44)$ \\ 
$1733$--$1777$ & $1783$--$1811$
& $1823$--$1867$ & $1871$--$1907$
& $1913$--$1951$ & $1973$--$1987$ \\ \hline
$(4,45)$ & $(4,46)$ & $(4,47)$
& $(4,48)$ & $(4,49)$ & \textcolor{red}{$(4,50)$} \\
$1993$--$2003$ & $2011$--$2039$ & $2053$--$2089$
& $2099$--$2141$ & $2143$--$2179$ & $2203$--$2221$ \\ \hline
\textcolor{red}{$(5,49)$} & \textcolor{red}{$(5,50)$}
& $(5,51)$ & $(5,52)$ & $(5,53)$ & $(5,54)$
\end{tabular}	
\end{center}
\caption{The optimal pair $(\xi,\eta)$ at which $G(s)$ reaches its minimum at $s=\frac{2\xi+1}{2p} -\frac{\eta}{p^2}$ for odd primes $p=3,\dots,2221$.}
\label{T:2221}
\end{table}

We list the partitions used in our numerical check for $p=127$,
$p=491$ and $p=1993$ for which $\xi$ jumps from $i$ to $i+1$, and the
minimum of $G$ is attained with the choices $(\xi,\eta)=(3,5)$,
$(4,15)$ and $(5,49)$, respectively. For simplicity, we use the
notation $[a,b]=[x_0,x_1,\dots,x_d]$ to mean the union 
$\cup_{i=1}^d [x_{i-1},x_i]$ with $x_0=a$ and $x_d=b$.
\begin{center}
\renewcommand{\arraystretch}{1.5}	
\begin{tabular}{cccc}
\multicolumn{4}{c}{$p=127$: $(\xi,\eta)=(3,5)$} \\ \hline
variable & 
Interval & \makecell{first \\partition} & \makecell{split each \\
into $N$ subintervals\\
(equal spacing)}\\ \hline
\multirow{2}{*}{$t$} &
$[0,\frac12-\frac1{2p}]$ & 
$[0,\frac12-\frac4p,\frac12-\frac1{2p}]$
& $N=500$ \\ &
$[\frac12+\frac1{2p},1]$ &
$[\frac12+\frac1{2p},\frac12+\frac9{2p},1]$ & $N=500$\\ \hline
\multirow{2}{*}{$s$} &
$[\frac1p,\frac{7}{2p}-\frac{11}{2p^2}]$ & 
$[\frac1p,\frac{2.3}{p}, \frac{3.4}{p}, 
\frac{7}{2p}-\frac{11}{2p^2}]$ 
& $N=1000$ \\  &
$[\frac7{2p}-\frac{9}{2p^2}, 1]$ & 
$[\frac7{2p}-\frac{9}{2p^2}, \frac{3.51}{p},
\frac6{p},1]$ & $N=1200$ \\ \hline
\end{tabular}	
\end{center}

\begin{center}
\renewcommand{\arraystretch}{1.5}	
\begin{tabular}{cccc}
\multicolumn{4}{c}{$p=491$: $(\xi,\eta)=(4,15)$} \\ \hline	
variable & 
Interval & \makecell{first \\partition} & \makecell{split each \\
into $N$ subintervals\\
(equal spacing)}\\ \hline
\multirow{2}{*}{$t$} &
$[0,\frac12-\frac1{2p}]$ & 
$[0,0.46,\frac12-\frac5{2p},\frac12-\frac1{2p}]$
& $N=1000$ \\ &
$[\frac12+\frac1{2p},1]$ &
$[\frac12+\frac1{2p},\frac12+\frac{2.1}{p}
,0.53,1	]$ & $N=1000$\\ \hline
\multirow{2}{*}{$s$} &
$[\frac1p,\frac{4.5}{p}-\frac{15.5}{p^2}]$ & 
$[\frac1p,\frac{3.4}{p},\frac{3.5}p, \frac{4.45}{p}
,\frac{4.5}{p}-\frac{15.5}{p^2}]$ 
& $N=5000$ \\  &
$[\frac{4.5}{p}-\frac{14.5}{p^2}, 1]$ & 
$[\frac{4.5}{p}-\frac{14.5}{p^2}, 
\frac{4.5}{p}-\frac{6}{p^2},\frac{5}{p},
\frac8{p},1]$ & $N=5000$ \\ \hline
\end{tabular}	
\end{center}

\begin{center}
\renewcommand{\arraystretch}{1.5}	
\begin{tabular}{cccc}
\multicolumn{4}{c}{$p=1993$: $(\xi,\eta)=(5,49)$} \\ \hline	
variable & 
Interval & \makecell{first \\partition} & \makecell{split each \\
into $N$ subintervals\\
(equal spacing)}\\ \hline
\multirow{2}{*}{$t$} &
$[0,\frac12-\frac1{2p}]$ & 
$[0,0.485,0.498,0.4994,\frac12-\frac1{2p}]$
& $N=4000$ \\ &
$[\frac12+\frac1{2p},1]$ &
$[\frac12+\frac1{2p},0.501,0.504,0.515,1	]$ & $N=3000$\\ \hline
\multirow{2}{*}{$s$} & $[\frac1p, \frac{5.5}{p}-\frac{49.5}{p^2}]$
& $[\frac1p,\frac{4.45}{p}, \frac{4.5}{p},
\frac{5.46}{p},\frac{5.474}{p}
,\frac{5.5}{p}-\frac{49.5}{p^2}]$ 
& $N=12000$ \\  &
$[\frac{5.5}{p}-\frac{48.5}{p^2}, 1]$ & 
$[\frac{5.5}{p}-\frac{48.5}{p^2},
\frac{5.5}{p}-\frac{41}{p^2}, 
\frac{5.6}{p}, \frac{15}{p},1]$ & $N=10000$ \\ \hline
\end{tabular}	
\end{center}

\subsection{Large $p$ asymptotics}
Assuming that $G$ reaches its minimum at $s=\hat{s}_p(\xi,\eta)$ 
for some $(\xi,\eta)$, 
so that the minimum value $\gb_p$ is  given by \eqref{E:Bml},
we give here some simple, not completely rigorous, estimates
of the two parameters $(\xi,\eta)$ and the minimum point $\hat s_p$
for a given large $p$.
We note first that
\begin{align}\label{rhopp}
  \rhop
=\frac{\log(\frac12p(p+1))}{\log p}
=2+ \frac{\log(1+\frac1p)-\log 2}{\log p}
=2-\frac{1}{\log_2p}+O\Bigpar{\frac{1}{p\log p}}.
\end{align}

We see from Table~\ref{T:2221} in \refS{SS:127+} that $\xi$ and $\eta$ both seem to
grow as $p$ grows (although not monotonically in case of $\eta$).
In fact, it is easy to see that at least $\xi+\eta$ must tend to infinity,
because otherwise there would be an infinite subsequence with some fixed
values of $\xi$ and $\eta$, but then it would follow from \eqref{E:Bml}
that as $p\to\infty$ along this subsequence, 
using $\rho_p\to2$ from \eqref{rhopp},
\begin{align}
  \gb_p=B_{\xi,\eta}\to \frac{\xi+1}{2\xi+1}>\frac12,
\end{align}
which contradicts \eqref{wilson0.5}.

We obtain  more precise estimates
by regarding $\xi$ and $\eta$ as continuous variables in \eqref{E:Bml}
and setting the partial derivatives of \eqref{E:Bml} with
respect to $\xi$ and $\eta$ equal to 0.
This yields the equations:
\begin{align}\label{partial=0}
	\begin{cases}\displaystyle
		\frac{(4\xi+3)p^2+(4\xi-4\eta+3)p+2\eta(2\eta-1)}
		{(\xi+1)(2\xi+1)p^2+(\xi+1)(2\xi-4\eta+1)p
		+2\eta(\xi+1)(2\eta-1)}
		-\frac{2\rhop p}{(2\xi+1)p-2\eta}=0,\\
		\displaystyle
		\frac{\rhop}{(2\xi+1)p-2\eta}
		-\frac{2p-4\eta+1}
		{(2\xi+1)p^2+(2\xi-4\eta+1)p+2\eta(2\eta-1)}=0.
	\end{cases}
\end{align}
The positive solution pair, say $(\xi_+,\eta_+)$, already gives a
very good approximation to the true values of $(\xi,\eta)$. Empirically, $(\lfloor \xi_++0.5\rfloor, \lfloor \eta_++0.45\rfloor)$
is identical to the true pair $(\xi,\eta)$ at which $G$ attains the
minimum for primes $p$ from $11$ to $79$, and differs by at most $1$ 
(at either $\xi$ or $\eta$ but not both) for primes up to $7907$. 
For $p$ as large as $p=1,000,003$, such a solution pair gives $(9,13206)$,
while the true minimum of $G$ is reached at $(\xi,\eta)=(9, 13203)$.

For large $p$, 
we may approximate the equations \eqref{partial=0} by ignoring all terms
that are  $O(\frac1p)$ or $O(\frac\eta{p\xi})$ times the leading terms in the various
numerators and 
denominators. (We assume that $\frac\eta{p\xi}$ is small.)
This gives the equations, using the notation $\gD:=\frac1p+\frac\eta{p\xi}$,
\begin{subequations}
  \begin{align}\label{asy2a}
  \frac{4\xi+3}{(\xi+1)(2\xi+1)}&=\frac{2\rhop}{2\xi+1}
\bigpar{1+O(\gD)},
\\\label{asy2b}
\frac{\rhop}{2\xi+1}&=\frac{2-\frac4p\eta}{2\xi+1}\bigpar{1+O(\gD)}
.\end{align}
\end{subequations}
If we further define $\gd_p:=2-\rhop\sim \frac1{\log_2p}$ (see \eqref{rhopp}), then
\eqref{asy2a} yields 
\begin{align}\label{asy3}
(4-2\gd_p)(\xi+1)
=2\rhop(\xi+1)
=(4\xi+3)\bigpar{1+O(\gD)}
=4\xi+3+O(\xi\gD).
\end{align}
Assuming $\xi,\eta=o(p)$, we have $\xi\gD=\frac{\xi+\eta}p=o(1)$ and then
\eqref{asy3} yields 
$2\gd_p\xi=1+o(1)$ and finally, using \eqref{rhopp},
\begin{align}\label{asy4}
  \xi\sim \frac{1}{2\gd_p}
\sim\frac{\log_2p}{2} =  \log_4 p.
\end{align}

Similarly, \eqref{asy2b} yields
\begin{align}\label{asy5}
2-\frac{4\eta}p=\rho_p +O(\gD) = 2-\gd_p+O(\gD)
\end{align}
leading to the empirical approximation for $\eta$:
\begin{align}\label{asy6}
\eta\sim \frac{1}{4}p\gd_p \sim \frac{p}{4\log_2p}.
\end{align}

The values \eqref{asy4} and \eqref{asy6} yield by \eqref{E:hat-sp}
the estimate for the minimum point 
\begin{align}\label{asy7}
  \hat s_p \sim \frac{\xi}{p}\sim\frac{\log_4 p}{p}.
\end{align}

Finally, observe that if
\begin{align}
	x = \sum_{j\geq1}b_{j}p^{-j} = (0.b_1b_2\dots)_p,
\end{align}
then $b_1 = \lfloor{px\rfloor}$, and 
\begin{align}
	b_m = \floor{p^mx} - p\floor{p^{m-1}x}
	= p\{p^{m-1}x\}-\{p^mx\}\qquad(m\ge2).
\end{align}
Thus
\begin{align}
	\varphi(x)
	&=\frac{1}{2}\sum_{j\geq1}\frac{
	\floor{p^jx} - p\floor{p^{j-1}x}}{A^j}
	\prod_{1\le i\le j}
	\left(1+\floor{p^ix} - p\floor{p^{i-1}x}\right)\notag\\
	&=\frac{1+\lfloor{px\rfloor}}2
	\left(\frac{\lfloor px\rfloor}{A}
	+ \sum_{j\geq2}\frac{
	p\{p^{j-1}x\}-\{p^jx\}}{A^j}
	\prod_{2\le i\le j}
	\left(1+p\{p^{i-1}x\}-\{p^ix\}\right)\right)
\end{align}
For large $p$, each term in the sum on the right-hand side is 
asymptotic to 
\begin{align}
  	\frac{2^{j-1}}{p^{j-1}}\,x \{px\}\cdots\{p^{j-2}x\}
	\{p^{j-1}x\}^2\lrpar{1+O\lrpar{p^{-1}}},
\end{align}
for $j\ge2$, and for $j=1$: 
\begin{align}
	\frac{\floor{px}(\floor{px}+1)}{p(p+1)}
	= x^2 + \frac{x(1-x-2\{px\})}{p}+O\lrpar{p^{-2}}.
\end{align}
We then obtain
\begin{align}
	\varphi(x) 
	= x^2 &+ \frac{x(1-x-2\{px\}+2\{px\}^2)}{p}
	+O\bigl(p^{-2}\bigr).
\end{align}
Then 
\begin{align}
	G(x) = x^{2-\rhop}
	+ \frac{x^{1-\rhop}(1-x-2\{px\}+2\{px\}^2)}{p}
	+\cdots,
\end{align}
where the piecewise differentiability of the terms on the right-hand side
might be  
useful in further identifying the true minimum of $G$ for large $p$.

\appendix

\section{The p-ary recurrence}\label{Srec}

In this appendix we study a more general recurrence, using the
methods of \cite{hk17} and \cite{hk24} where binary recurrences are
studied; see also Section 7.1 in the
\href{https://arxiv.org/abs/2210.10968}{earlier version of
\cite{hk24}} on arXiv.

Let $p$ be any integer larger than $1$. Consider the recurrence 
\begin{equation}
	f(n)=\sum_{0\leq j<p}\gamma_{j}
	f\left(\left\lfloor \frac{n+j}{p}\right\rfloor \right)
	\qquad\textrm{for \ensuremath{n\geq p}},
	\label{rec1}
\end{equation}
with given coefficients $\gamma_{0},\ldots,\gamma_{p-1}$ and given
initial values $f(1),\ldots,f(p-1)$ . We assume for simplicity that
$\gamma_{0},\ldots,\gamma_{p-1}>0$. Let
\begin{equation}\label{eq:A}
	A
	:=\sum_{0\leq j<p}\gamma_{j}. 
\end{equation}

\begin{lemma}\label{L1}
Let $\gamma_{0},\ldots,\gamma_{p-1}>0$. Then there exists a unique
strictly increasing continuous function $\varphi$ on $[0,1]$ such
that $\varphi(0)=0,\varphi(1)=1$ and for $j=0,1,\ldots,p-1$,
\begin{equation}
	\varphi(t)
	=\frac{\gamma_{p-j-1}}{A}\varphi(pt-j)
	+\frac{\sum_{p-j\leq i<p}\gamma_i}{A}
	\qquad\mathrm{if}\quad
	\frac{j}{p}\leq t\leq\,\frac{j+1}{p}.
	\label{l1}
\end{equation}
Moreover, we have the explicit formula
\begin{equation}
	\varphi\left(\sum_{j\geq1}b_{j}p^{-j}\right)
	=\sum_{j\geq1}A^{-j}
	\left(\sum_{p-b_{j}\leq i<p}\gamma_i\right)
	\left(\prod_{1\le i<j}\gamma_{p-1-b_i}\right),
	\label{E2}
\end{equation}
when $b_{j}\in\{0,1,2,\ldots,p-1\}$ for $j\ge1$.
\end{lemma}

\begin{proof}
Define $\varphi_{0}(t):=t$ for $t\in[0,1]$, and recursively let
\begin{equation}
	\varphi_{k+1}(t)
	:=\frac{\gamma_{p-j-1}}{A}\varphi_{k}(pt-j)
	+\frac{\sum_{p-j\leq i<p}\gamma_i}{A}
	\qquad\mathrm{if}\quad\frac{j}{p}\leq t\leq\,\frac{j+1}{p},
	\label{l1k}
\end{equation}
for $j=0,1,\ldots,p-1$. (Thus $\varphi_{k+1}$ consists of $p$
suitably scaled copies of $\varphi_{k}$.) Observe first that
$\gf_k(0)=0$ and $\gf_k(1)=1$ by induction. Note also that if
$t_{0}=\frac{j_{0}}{p}$ for $j_{0}\in\{1,\ldots,p-1\}$, then the
definition \eqref{l1k} can be applied with both $j=j_0-1$ and
$j=j_0$. The first choice gives
\begin{equation}
	\varphi_{k+1}\left(t_{0}\right)
	=\frac{\gamma_{p-j_{0}}}{A}\varphi_{k}(1)
	+\frac{\sum_{p-j_{0}+1\leq i<p}\gamma_i}{A}
	=\frac{\sum_{p-j_{0}\leq i<p}\gamma_i}{A},
\end{equation}
and the second one gives
\begin{equation}
	\varphi_{k+1}\left(t_{0}\right)
	=\frac{\gamma_{p-j_{0}-1}}{A}\varphi_{k}(0)
	+\frac{\sum_{p-j_{0}\leq i<p}\gamma_i}{A}
	=\frac{\sum_{p-j_{0}\leq i<p}\gamma_i}{A}.
\end{equation}
Since these are equal, the definition (\ref{l1k}) is consistent. It
is now obvious by induction that $\varphi_{k}$ is continuous and
strictly increasing.

We claim that
\begin{equation}
	\left|\varphi_{k+1}(t)-\varphi_{k}(t)\right|
	\leq\left(\frac{\max_{0\leq j<p}\gamma_{j}}{A}\right)^{k}
	\qquad\textrm{for all \ensuremath{k\geq0} and }t\in[0,1].
	\label{l1c}
\end{equation}
We prove this by induction. The case $k=0$ is clear, since
$\left|\varphi_1(t)-\varphi_{0}(t)\right|\leq1$. Assume now that 
(\ref{l1c}) holds for $k-1$. If $\frac{j}{p}\leq t\leq\, 
\frac{j+1}{p}$ then
\begin{equation}
	\begin{aligned}[b]
		\left|\varphi_{k+1}(t)-\varphi_{k}(t)\right| 
		&    =\frac{\gamma_{p-j-1}}{A}
		\left|\varphi_{k}(pt-j)-\varphi_{k-1}(pt-j)\right|\\
		& \leq\frac{\gamma_{p-j-1}}{A}
		\left(\frac{\max_{0\leq j<p}\gamma_{j}}{A}\right)^{k-1}\\
		& \leq\left(\frac{\max_{0\leq j<p}\gamma_{j}}{A}\right)^{k}.
	\end{aligned}
\end{equation}
Thus, (\ref{l1c}) holds, and since $\max_{0\leq j<p}\gamma_{j}/A<1$,
it follows that the sequence $\varphi_{k}$ converges uniformly to a
function $\varphi:\oi\to\oi$. By (\ref{l1k}), $\varphi$ satisfies
(\ref{l1}). Since each $\varphi_{k}$ is continuous and strictly
increasing, the limiting function $\varphi$ is continuous and
non-decreasing.

We next prove by induction that 
\begin{equation}
	\varphi_{N}^{}\left(\sum_{1\leq j\leq N}b_{j}p^{-j}\right)
	=\sum_{1\leq j\leq N}A^{-j}
	\left(\sum_{p-b_{j}\leq i<p}\gamma_i\right)
	\left(\prod_{1\le i<j}\gamma_{p-1-b_i}\right),
	\label{E1}
\end{equation}
for any $N\ge0$, where $b_{j}\in\{0,1,2,\ldots,p-1\}$. This is
trivial for $N=0$. Suppose that \eqref{E1} holds for $N$. Let
$t=\sum_{1\leq j\leq N+1}b_{j}p^{-j}$, where
$b_{j}\in\{0,1,2,\ldots,p-1\}$. Then, by \eqref{l1k},
	\begin{align}
		\varphi_{N+1}^{}(t) 
		&=\frac{\gamma_{p-b_1-1}}{A}
		\varphi_{N}\left(pt-b_1\right)
		+\frac{\sum_{p-b_1\leq i<p}\gamma_i}{A}\notag\\
		&=\frac{\gamma_{p-b_1-1}}{A}
		\left(\sum_{1\leq j\leq N}A^{-j}
		\left(\sum_{p-b_{j+1}\leq i<p}\gamma_i\right)
		\left(\prod_{1\le i<j}\gamma_{p-1-b_{i+1}}\right)\right)
		+\frac{\sum_{p-b_1\leq i<p}\gamma_i}{A}\notag\\
		&=\sum_{2\leq j\leq N+1}A^{-j}
		\left(\sum_{p-b_{j}\leq i<p}\gamma_i\right)
		\left(\prod_{1\le i<j}\gamma_{p-1-b_i}\right)
		+\frac{\sum_{p-b_1\leq i<p}\gamma_i}{A}\notag\\
		&=\sum_{1\leq j\leq N+1}A^{-j}
		\left(\sum_{p-b_{j}\leq i<p}\gamma_i\right)
		\left(\prod_{1\le i<j}\gamma_{p-1-b_i}\right).
	\end{align}
Hence \eqref{E1} holds for $N+1$, and thus it holds in general by 
induction.

For any $p$-adic rational $t=\sum_{1\leq j\leq M}b_{j}p^{-j}$, let
$b_j:=0$ for $j>M$ and apply \eqref{E1} with $N\ge M$. Letting
$N\to\infty$, we see that \eqref{E2} holds for $t$. Since we have
shown that $\gf$ is continuous, it follows that \eqref{E2} holds in
general, for any $\sum_{1\leq j<\infty}b_{j}p^{-j}$. Furthermore, it 
follows from \eqref{E1} that for every $N\ge M$, we have
\begin{equation}
	\varphi_{N}^{}(t)
	=\varphi_{M}^{}(t).
\end{equation}
Thus $\varphi(t) =\lim_{N\rightarrow\infty} \varphi_{N}(t) =\varphi_{M}(t)$. Accordingly, 
\begin{equation}
	\varphi(t)
	=\varphi_{N}(t)
	<\varphi_{N}(t+p^{-N})
	=\varphi(t+p^{-N}).
\end{equation}
This shows that $\varphi$ is strictly increasing on $p$-adic
rationals. In general, for $0\leq s_1<s_2\leq1$, there exist
$p$-adic rationals $t_1$ and $t_2$ such that $s_1\leq
t_1<t_2\leq s_2$. Then
\begin{equation}
	\varphi(s_1)
	\leq\varphi(t_1)
	<\varphi(t_2)
	\leq\varphi(s_2).
\end{equation}
Consequently, $\varphi$ is strictly increasing. 
\end{proof}

We next extend $f(n)$ to a function of a real variable $x\geq1$ by 
\begin{equation}
	f(n+t)
	:=(1-\varphi(t))f(n)+\varphi(t)f(n+1)
	\label{real}
\end{equation}
for $n\geq1$ and $0\leq t\leq1$. 

\begin{lemma}\label{L2}
Assume that the recurrence \eqref{rec1} holds. Then
\begin{equation}
	f(x)
	=Af\left(\frac{x}{p}\right)
	\qquad\textrm{for all real }x\geq p.
	\label{l2a}
\end{equation}
\end{lemma}
\begin{proof}
Define 
\begin{equation}
	A_{k}^{-}
	:=\sum_{0\leq i<k}\gamma_i
	\qquad\textrm{and}\qquad 
	A_{k}^{+}
	:=\sum_{k\leq i<p}\gamma_i.
\end{equation}
Rewrite (\ref{rec1}) as 
\begin{equation}
	f(pn+j)=A_{p-j}^{-}f(n)+A_{p-j}^{+}f(n+1)
	\label{l2b}
\end{equation}
for $n\geq1$ and $j=0,1,\ldots,p-1$; note that for $j=0$ \eqref{l2b}
is $f(pn)=Af(n)$, and it follows that \eqref{l2b} holds for $j=p$
too. Also rewrite \eqref{l1} as
\begin{equation}
	\gamma_{p-1-j}\varphi(pt)
	=A\varphi\left(\frac{j}{p}+t\right)-A_{p-j}^{+}
	\label{l2c}
\end{equation}
for $0\leq t\leq\frac{1}{p}$ and $j=0,1,\ldots,p-1$. 

Now, for $x\geq p$, write 
\begin{equation}
	x=pn+j+pt,
\end{equation}
where 
\begin{equation}
	n
	=\left\lfloor \frac{x}{p}\right\rfloor ,
	\quad 
	j
	=\left\lfloor x\right\rfloor \ \mathrm{mod \ } p,
	\quad t=\frac{\{x\}}{p}.
\end{equation}
Then, by (\ref{real}), (\ref{l2b}) and (\ref{l2c}),
\begin{align}
	f(x) &
	=f(pn+j+pt) 
	\notag\\
	&=(1-\varphi(pt))f(pn+j)+\varphi(pt)f(pn+j+1)
	\notag\\
	&=(1-\varphi(pt))\left(A_{p-j}^{-}f(n)
	+A_{p-j}^{+}f(n+1)\right)\\
	&\qquad\quad +\varphi(pt)
	\left(A_{p-j-1}^{-}f(n)+A_{p-j-1}^{+}f(n+1)\right)
	\notag\\
	&=\left(A_{p-j}^{-}f(n)+A_{p-j}^{+}f(n+1)\right)
	-\varphi(pt)\gamma_{p-j-1}\left(f(n)-f(n+1)\right)
	\notag\\
	&=\left(A_{p-j}^{-}f(n)+A_{p-j}^{+}f(n+1)\right)
	-\left(A\varphi\left(\frac{j}{p}+t\right)
	-A_{p-j}^{+}\right)\left(f(n)-f(n+1)\right)
	\notag\\
	&=Af(n)-A\varphi\left(\frac{j}{p}+t\right)f(n)
	+A\varphi\left(\frac{j}{p}+t\right)f(n+1)
	\notag\\
	&=  A\left(\left(1-\varphi\left(\frac{j}{p}+t\right)\right)f(n)
	+\varphi\left(\frac{j}{p}+t\right)f(n+1)\right)
	\notag\\
	&=Af\left(n+\frac{j}{p}+t\right)
	\notag\\
	&=Af\left(\frac{x}{p}\right),
\end{align}
which proves \eqref{l2a}.
\end{proof}

\begin{theorem}\label{TP}
Assume that the recurrence \eqref{rec1} holds, with
$\gamma_0,\dots,\gamma_{p-1}>0$. Then
\begin{equation}\label{tpa}
	f(n)
	=n^{\varrho}\mathcal{P}\left(\log_pn\right)
	\qquad\textrm{for all }n\geq1,
\end{equation}
where $\varrho:=\log_pA$ and
\begin{equation}\label{tpb}
	\mathcal{P}(t):=A^{-\{t\}}f(p^{\{t\}})
\end{equation}
is a continuous $1$-periodic function. 

Moreover, if the initial values satisfy
\begin{equation}
	f(j)
	=\sum_{p-j\leq i<p}\gamma_i
	\qquad\textrm{for }j=1,\ldots,p-1
	\label{d1}
\end{equation}
then 
\begin{equation}\label{tpd}
	\mathcal{P}(t)
	=A^{1-\{t\}}\varphi(p^{\{t\}-1}),
\end{equation}
where $\varphi$ is defined in \refL{L1}.
\end{theorem}

\begin{proof}
Since $f(x)$ is continuous, $\cP(t)$ is continuous on $[0,1)$, and by
(\ref{l2a})
\begin{equation}
	\lim_{t\nearrow1}\mathcal{P}(t)
	=A^{-1}f(p)=f(1)=\mathcal{P}(0)
	=\mathcal{P}(1),
\end{equation}
which shows that $\mathcal{P}(t)$ is a continuous 1-periodic 
function. For $y\in[1,p)$
\begin{equation}
	f(y)
	=f(p^{\log_py})
	=A^{\log_py}\left(A^{-\log_py}f(p^{\log_py})\right)
	=A^{\log_py}\mathcal{P}\left(\log_py\right),
\end{equation}
and, for each $x\geq1$, 
\begin{equation}
	p^{-\lfloor \log_px\rfloor}x\in[1,p).
\end{equation}
By applying \refL{L2} repeatedly $\lfloor \log_px\rfloor$ times:
\begin{equation}
\begin{split}
	f(x)
	&=A^{\lfloor \log_px\rfloor}
	f\left(p^{-\lfloor \log_px\rfloor}x\right)\\
	&=A^{\lfloor \log_px\rfloor+\log_p
	\left(p^{-\lfloor \log_px\rfloor}x\right)}
	\mathcal{P}\left(\log_p
	\left(p^{-\lfloor \log_px\rfloor}x\right)\right)\\
	&=A^{\log_px}\mathcal{P}\left(\log_px\right).
\end{split}	
\end{equation}
Thus we get
\begin{equation}
	f(n)
	=A^{\log_pn}\mathcal{P}\left(\log_pn\right)
	=p^{\varrho\log_pn}\mathcal{P}\left(\log_pn\right)
	=n^{\varrho}\mathcal{P}\left(\log_pn\right),
\end{equation}
proving \eqref{tpa}.

Finally, suppose that condition (\ref{d1}) holds. For $1\leq x<p$, let
$j=\left\lfloor x\right\rfloor $. Then, using \eqref{real} and \eqref{l1}, 
\begin{equation}
\begin{aligned}[b]
	f(x) 
	& =f(j)+\varphi(x-j)\left(f(j+1)-f(j)\right)\\
	& =\sum_{p-j\leq i<p}\gamma_i
	+\varphi(x-j)\gamma_{p-j-1}\\
	& =A\left(\frac{\gamma_{p-j-1}}{A}\varphi(x-j)
	+\frac{\sum_{p-j\leq i<p}\gamma_i}{A}\right)\\
	& =A\varphi\left(\frac{x}{p}\right).
\end{aligned}
\end{equation}
Thus we have
\begin{equation}
	\mathcal{P}(t)
	=A^{-\{t\}}f(p^{\{t\}})=A^{1-\{t\}}\varphi(p^{\{t\}-1}).
    \qedhere
\end{equation}
\end{proof}

\begin{corollary}
Assume that the recurrence \eqref{rec1} holds, with
$\gamma_0,\dots,\gamma_{p-1}>0$. Then
\begin{align}
	\sup_{n\ge1}\frac{f(n)}{n^{\varrho}}&
	=\limsup_{n\rightarrow\infty}\frac{f(n)}{n^{\varrho}} 
	=\max_{t\in\oi}\mathcal{P}(t),
	\\
	\inf_{n\ge1}\frac{f(n)}{n^{\varrho}}&
	=\liminf_{n\rightarrow\infty}\frac{f(n)}{n^{\varrho}} 
	=\min_{t\in\oi}\mathcal{P}(t).
\end{align}
\end{corollary}
\begin{proof}
By \eqref{tpa}, since $\cP(t)$ is a continuous 1-periodic function.
\end{proof}
\begin{remark}\label{Rinitial}
In the case that $\gamma_{p-1}=1$, the condition \eqref{d1} on the
initial values is equivalent to assuming that the recurrence
\eqref{rec1} extends to all $n\ge2$, with $f(0)=0$ and $f(1)=1$.
\end{remark}

\bibliographystyle{amsplain}

\end{document}